\documentclass[11pt]{article}
\usepackage{lmodern}    
\usepackage{amsfonts}
\usepackage{amsmath}
\usepackage{amsxtra}
\usepackage{amsthm}
\usepackage[shortlabels]{enumitem}
\usepackage{mathrsfs}
\usepackage{stmaryrd}
\usepackage{dsfont}
\usepackage{bbm}
\usepackage{graphicx}
\usepackage{fullpage}
\usepackage[dvipsnames]{xcolor}
\usepackage[colorlinks,linkcolor=RedViolet,citecolor=PineGreen,anchorcolor=Green,urlcolor=BlueViolet]{hyperref}
\usepackage[nottoc]{tocbibind}
\usepackage[misc]{ifsym}

\makeatletter
\newcommand{\myitem}[1]{%
\item[#1]\protected@edef\@currentlabel{#1}%
}
\makeatother

\makeatletter
\newcommand{\optionaldesc}[2]{%
  \phantomsection
  #1\protected@edef\@currentlabel{#1}\label{#2}%
}
\makeatother

\makeatletter
\newcommand*{\bigcdot}{}
\DeclareRobustCommand*{\bigcdot}{%
  \mathbin{\mathpalette\bigcdot@{}}%
}
\newcommand*{\bigcdot@scalefactor}{.7}
\newcommand*{\bigcdot@widthfactor}{1.15}
\newcommand*{\bigcdot@}[2]{%
  \sbox0{$#1\vcenter{}$}
  \sbox2{$#1\cdot\m@th$}%
  \hbox to \bigcdot@widthfactor\wd2{%
    \hfil
    \raise\ht0\hbox{%
      \scalebox{\bigcdot@scalefactor}{%
        \lower\ht0\hbox{$#1\bullet\m@th$}%
      }%
    }%
    \hfil
  }%
}
\makeatother

\renewenvironment{thebibliography}[1]{%
\begin{oldthebibliography}{#1}%
\setlength{\baselineskip}{.9em}
\linespread{.99}
\setlength{\parskip}{0ex}%
\setlength{\itemsep}{.4em}%
}%
{%
\end{oldthebibliography}%
}
\newtheorem{theorem}{Theorem}[section]
\newtheorem{proposition}[theorem]{Proposition}

\newtheorem{lemma}[theorem]{Lemma}

\theoremstyle{definition}
\newtheorem{definition}[theorem]{Definition}
\newtheorem{remark}[theorem]{Remark}
\newtheorem{example}[theorem]{Example}

\newtheorem*{acknowledgement}{Acknowledgements}

\DeclareMathOperator{\cl}{cl}
\DeclareMathOperator{\essinf}{ess\ inf}

\DeclareMathOperator{\argmin}{argmin\, }
\DeclareMathOperator{\Var}{Var}
\DeclareMathOperator{\conv}{conv}
\DeclareMathOperator{\Log}{\mathcal{L}}
\newcommand{\Exp}{\mathscr{E}}
\newcommand{\e}{\mathrm{e}}
\renewcommand{\d}{\mathrm{d}}
\newcommand{\Inf}{\inf\limits}
\newcommand{\Lim}{\lim\limits}
\newcommand{\one}{\mathbbm{1}}
\newcommand{\sint}{\bigcdot}

\newcommand{\R}{\mathbb{R}}
\newcommand{\N}{\mathbb{N}}

\newcommand{\F}{\mathscr{F}}
\newcommand{\Gs}{\mathscr{G}}

\newcommand{\tauE}{{}^\tau\!\Exp}
\newcommand{\TmE}{{}^{T_m}\!\Exp}
\newcommand{\TkE}{{}^{T_k}\!\Exp}
\newcommand{\tauZ}{{}^\tau\!Z}
\newcommand{\tauN}{{}^\tau\!N}
\newcommand{\sigZ}{{}^{\sigma}\!Z}
\newcommand{\sigmZ}{{}^{\sigma_m}\! Z}
\newcommand{\sigmkZ}{{}^{\sigma_{m_k}}\! Z}
\newcommand{\tauhZ}{{}^\tau\!\hat{Z}}

\newcommand{\TB}{\overline{\Theta}}

\newcommand{\TBtau}{\overline{\Theta}_{\tau}}

\newcommand{\TT}{\widetilde{\Theta}}

\newcommand{\TTtau}{\widetilde{\Theta}_{\tau}}

\newcommand{\GH}{\widehat{\mathcal{G}}}
\newcommand{\GB}{\overline{\mathcal{G}}}
\newcommand{\GBtau}{\overline{\mathcal{G}}_{\tau}}
\newcommand{\GT}{\widetilde{\mathcal{G}}}
\newcommand{\GTtau}{\widetilde{\mathcal{G}}_{\tau}}
\newcommand{\cG}{\mathcal{G}}
\newcommand{\cGtau}{\mathcal{G}_\tau}
\newcommand{\vt}{\vartheta}

\newcommand{\vp}{\varphi}

\newcommand{\ve}{\varepsilon}

\newcommand{\E}{\mathrm{E}}
\newcommand{\T}{\mathcal{T}}

\numberwithin{equation}{section}

\begin{document}
\title{The law of one price in quadratic hedging\\ and mean--variance portfolio selection}
\author{Ale\v{s} \v{C}ern\'{y}\footnote{Bayes Business School, City St George's, University of London, 106 Bunhill Row, London EC1Y 8TZ, UK, {\tt ales.cerny.1@city.ac.uk}.}\and Christoph Czichowsky\footnote{Department of Mathematics, London School of Economics and Political Science, Columbia House, Houghton Street, London WC2A 2AE, UK, {\tt c.czichowsky@lse.ac.uk} ({\scriptsize\Letter}). }\raisebox{0.7ex}{\scriptsize\, ,\,\Letter}}
\date{First version: October 28, 2022; this version: \today}
\maketitle
\vspace{-0.5cm}
\begin{center}
\emph{Dedicated to Martin Schweizer on the occasion of his 60th birthday}\bigskip\medskip
\end{center}
\begin{abstract} 
The \emph{law of one price (LOP)} broadly asserts that identical financial flows should command the same price. We show that, when properly formulated,  LOP is the minimal condition for a well-defined mean--variance portfolio selection framework without degeneracy.  Crucially, the paper identifies a new mechanism through which LOP can fail in a continuous-time $L^2$ setting without frictions, namely `trading from just before a predictable stopping time', which surprisingly identifies LOP violations even for continuous price processes.\smallskip 

\noindent Closing this loophole allows to give a version of the ``Fundamental Theorem of Asset Pricing'' appropriate in the quadratic context, establishing the equivalence of the economic concept of LOP with the probabilistic property of the existence of a local $\Exp$-martingale state price density. The latter provides unique prices for all square-integrable claims in an extended market and subsequently plays an important role in quadratic hedging and mean--variance portfolio selection.\smallskip 

\noindent Mathematically, we formulate a novel variant of the uniform boundedness principle for conditionally linear functionals on the $L^0$ module of conditionally square-integrable random variables. We then study the representation of time-consistent families of such functionals in terms of stochastic exponentials of a fixed local martingale.\medskip

 \medskip

\medskip

\end{abstract}

\newpage
\section{Introduction}
In this paper, we take a fresh look at the mathematical foundations of quadratic hedging. For a fixed time horizon $T>0$, a locally square-integrable semimartingale price process $S$, and a square-integrable contingent claim $H$, the goal of quadratic hedging is to 
\begin{equation}\label{MVH1}
\text{minimise $\E\left[\left(v+\int_{(0,T]}\vt_u dS_u-H\right)^2\right]$} 
\end{equation}
over initial wealth $v\in\R$ and all reasonable trading strategies $\vt$; see Pham \cite{pham.00} and Schweizer \cite{schweizer.01,schweizer.10} for literature overviews. The task \eqref{MVH1}, first formulated and analyzed in the seminal paper of Schweizer \cite{schweizer.92.aap},  was abstracted from an early work on mean--variance hedging by Duffie and Richardson \cite{duffie.richardson.91}; it is thus intimately related to efficient portfolio selection in the sense of Markowitz \cite{markowitz.52}.%

Our point of departure is the quadratic hedging framework of \v{C}ern\'{y} and Kallsen~\cite{cerny.kallsen.07}, or rather its extension by Czichowsky and Schweizer~\cite{czichowsky.schweizer.13} to settings that admit some arbitrage opportunities, properly formulated as `$L^2$ free lunches' (Schachermayer~\cite[Definition~1.3]{schachermayer.02}). The aim is to obtain minimal conditions on $S$ and $\vt$, under which the solution of \eqref{MVH1} exists and retains the form reported in \cite{cerny.kallsen.07,czichowsky.schweizer.13}. This leads us to study an appropriate version of the \emph{law of one price (LOP)}  and its implication for the existence of suitably modified state price densities in an $L^2$ setting.

Our analysis adds an important qualifier to the classical textbook folklore asserting that the existence of a non-zero state price density is sufficient for the law of one price in continuous-time models; cf. Cochrane \cite[Section~4.3]{cochrane.01}. We show that it is not enough to consider signed conditional state price densities as in the seminal paper of Hansen and Richard \cite{hansen.richard.87} but one must also require certain limiting properties at predictable stopping times. Our main result (Theorem \ref{T:M1}) provides a version of the ``Fundamental Theorem of Asset Pricing'' appropriate in the quadratic context, establishing the equivalence of the economic concept of LOP with the probabilistic property of the existence of a local $\Exp$-martingale state price density (stemming from the notion of $\Exp$-martingales introduced by Choulli, Krawczyk, and Stricker \cite{choulli.al.98}, which we extend slightly here). The latter gives unique prices for all square-integrable contingent claims in an extended market and can be chosen such that the passage to the extended market does not alter the efficient frontier (Remark~\ref{R:VOED}). 
Mathematically, we formulate a novel variant of the uniform boundedness principle for conditionally linear functionals on the $L^0$ module of conditionally square-integrable random variables (Proposition \ref{P:UBP}). We then study the representation of time-consistent families of such functionals in terms of stochastic exponentials of a fixed local martingale (Proposition \ref{P:sopZ}).

There are very few theoretical studies of LOP in frictionless markets and none specifically in the context of quadratic hedging. On closer reading, the existing studies are also limited in scope. Courtault, Delbaen, Kabanov, and Stricker \cite{courtault.al.04} examine finite discrete-time models, observing on p. 529  
\begin{quote}
\emph{\textellipsis [the discrete-time] results have no natural counterparts for continuous-time models. ``Natural'' here means ``for the standard concept of admissibility''. The latter requires that the value process is bounded from below.}
\end{quote}
B\"attig and Jarrow~\cite{battig.jarrow.99} study wealth transfers between two fixed dates in a complete market setting. Their assumption of a signed pricing measure in \cite[Theorem~1]{battig.jarrow.99} implicitly invokes the law of one price on $L^\infty$ with proximity given by the weak-star topology. One should observe that the space $L^\infty(P)$ depends on $P$ only through the null sets, while $L^2(P)$ lacks such invariance property. This paper provides, for the first time in the literature, a complete analysis of the law of one price in a continuous-time $L^2$ setting without frictions. The repercussion of our findings for the arbitrage theory on $L^\infty$ are left for future work. For continuous price processes, we show that LOP is equivalent to the existence of an equivalent local martingale measure with square-integrable density and hence to the condition of no `$L^2$ free lunch' as shown in Stricker \cite[Theorems~2 and 3]{stricker.90}. In Example~\ref{example}, we illustrate that LOP can fail for continuous price processes that satisfy the no-arbitrage (NA) condition both in the $L^2$ and the $L^\infty$ sense (Schachermayer~\cite[Definition~1.1]{schachermayer.02}).

The paper is organised as follows. In Section~\ref{S:2}, after establishing notation (Subsection~\ref{SS:2.1}), market dynamics, and admissible strategies (Subsection~\ref{SS:2.2}),  we introduce alternative descriptions of the law of one price by means of (i) the price process $S$; (ii) pricing functionals; and (iii) state price densities (Subsections~\ref{SS:2.3}--\ref{SS:2.6}). We conclude Section~\ref{S:2} with a review  of $\Exp$-densities and $\Exp$-martingales (Subsection~\ref{SS:2.7}). Section~\ref{S:3} contains the main results of the paper. In Subsection~\ref{SS:3.1}, we pull the various notions of the law of one price together to demonstrate their equivalence (Theorem~\ref{T:M1}). Subsection~\ref{SS:3.2}  illustrates several phenomena that arise when the law of one price fails. In Subsection~\ref{SS:3.3bis} we offer some intuition for the law of one price and interpret the main result (Theorem~\ref{T:M1}) as a market extension theorem. We next examine the consequences of LOP for quadratic hedging (Subsection~\ref{SS:3.4}), extend its applicability to the conditional framework of Hansen and Richard \cite{hansen.richard.87} (Subsection~\ref{SS:3.5}), and study the resulting mean--variance portfolio selection in the presence of a contingent claim (Subsection~\ref{SS:3.6}). Section~\ref{S:4} contains proofs of the main theorem presented via several partial statements of independent interest.


\section{Problem formulation}\label{S:2}


\subsection{Preliminaries}\label{SS:2.1} 

We work on a filtered probability space $\big(\Omega,\F,(\F_t)_{0\leq t\leq T},P\big)$ with $\F_T=\F$, satisfying the usual conditions of right continuity and completeness. For $p\in[0,\infty]$, we shall frequently write $L^p(P)$ or just $L^p$ as a shorthand for $L^p(\F_T,P)$. The $L^2$ closure of an arbitrary $\mathcal{A}\subseteq L^2$ is denoted by $\cl\mathcal{A}$.\medskip 

\noindent{\bf Conditional expectations.} Throughout the paper, we consider generalised conditional expectations as defined in \cite[I.1.1]{js.03}, for example. That is, for a random variable $X$ and $\Gs\subseteq\F$, the conditional expectation $\E[X\mkern2mu|\mkern2mu\Gs]$ is finite if and only if there exists a $\Gs$-measurable random variable $K>0$ such that $KX$ is integrable (He, Wang, and Yan \cite[Theorem~1.16]{he.wang.yan.92}), in which case one has $\E[X\mkern2mu|\mkern2mu\Gs]=\frac{\E[KX\mkern2mu|\mkern2mu\Gs]}{K}$, where the right-hand side now features the standard conditional expectation.
In particular, $X$ may have finite conditional expectation without being integrable.\medskip

\noindent{\bf Stopping times.} We use the convention that $\inf\emptyset=\infty$ so that stopping times are, in general, $[0,T]\cup\{\infty\}$-valued. The set of all $[0,T]$-valued stopping times is denoted by $\T$. A sequence of stopping times $(\sigma_n)_{n=1}^\infty$ is said to converge \emph{stationarily} to $T$ if $P(\sigma_n<T)\to 0$ \cite[p.~297]{dellacherie.meyer.82}. Let $\mathscr{C}$ be a class of stochastic processes. With Dellacherie and Meyer \cite[Definition~V.27]{dellacherie.meyer.82}, we say that $X$ belongs to $\mathscr{C}$ \emph{locally} if there is a sequence of stopping times $\tau_n$ increasing to $\infty$ (a \emph{localizing sequence}) such that $X^{\tau_n}\one_{\{\tau_n>0\}}\in\mathscr{C}$ for each $n\in\N$. This notion of localization is slightly broader than the one in Jacod and Shiryaev \cite[I.1.34]{js.03}. Note that if $(\tau_n)_{n=1}^\infty$ is a localizing sequence of stopping times, the stopping times $\tau_n\wedge T$ converge stationarily to $T$. Thus on the finite time horizon $[0,T]$ assumed in this paper, it is sufficient to consider localizing sequences that converge stationarily to $T$.\medskip

\noindent{\bf Stochastic exponential and logarithm.} For a semimartingale $S$, we denote by $L(S)$ the space of all $S$-integrable predictable processes $\vt=(\vt_t)_{0\leq t\leq T}$ and write $\vt\sint S_t:=\int_{(0,t]}\vt_u dS_u$ for their stochastic integral up to time $t\in[0,T]$. The symbol $\Exp(N)$ denotes the stochastic exponential of a semimartingale $N$, i.e., the unique strong solution of the stochastic differential equation $\Exp(N) = 1+\Exp(N)_-\sint N$; see \cite{doleans-dade.70}. We say that $S$ \emph{does not reach zero continuously and is absorbed in zero} if for $\sigma^S:=\inf\{ t>0\,:\,S_t=0\}$ one has
\begin{equation}\label{eq:dnrzcaz}
\text{ $S_-\neq 0$ on $\llbracket 0,\sigma^S\rrbracket$ and $S=0$ on $\llbracket \sigma^S,\infty\llbracket$.} 
\end{equation}
 For such $S$ the stochastic logarithm $\Log(S)$ is well defined by  
$ \Log(S) = \frac{\one_{\llbracket 0,\sigma^S\rrbracket}}{S_-}\sint S $
and has the property $S = S_0\Exp(\Log(S))$; see \cite[Proposition 2.2]{choulli.al.98}.\medskip

\noindent{\bf Further semimartingale notation.} For a special semimartingale $Y$, we write $Y=Y_0+M^Y+B^Y$ for the canonical decomposition of $Y$ into its local martingale component $M^Y$ and  its finite variation predictable component $B^Y$, both starting at zero. We denote by $\mathcal{H}^2$ the class of all special semimartingales $Y$ such that $\E[[M^Y,M^Y]_\infty]+\E[(\int_0^\infty|dB^Y_u|)^2]<\infty$; see Protter \cite[Chapter IV]{protter.05}.


\subsection{Asset prices and trading strategies}\label{SS:2.2}

Consider a financial market consisting of one riskless asset with constant value $1$ and $d$ risky assets described by an $\R^d$-valued locally square-integrable semimartingale $S$. Denote the running supremum of $|S|$ by $S^*$.
\begin{definition}\label{D:simple}
A trading strategy $\vt$ is called \emph{simple} if it is of the form $\vartheta=\sum_{i=1}^{m-1}\xi_i \one_{\rrbracket \sigma_i,\sigma_{i+1}\rrbracket }$ with stopping times \hbox{$0\leq\sigma_1\leq\cdots\leq\sigma_m\le \tau_n$} for some $n\in\N$, some bounded $\R^d$-valued $\F_{\sigma_i}$-measurable random variables $\xi_i$ for $i=1,\ldots,m-1$, and some localizing sequence $(\tau_n)_{n=1}^\infty$ of stopping times converging stationarily to $T$ such that $S^*_{\tau_n}\in L^2$. The set of all simple trading strategies is denoted by $\Theta$. 

We further introduce strategies that trade from a stopping time $\tau$ or from just before a predictable stopping time $\tau$ by setting
\begin{align*}
\Theta_\tau&{}:=\{\vt\in\Theta\,:\,\vt\one_{\llbracket0,\tau\rrbracket}=0\},\qquad{\tau\in\T};\\
\Theta_{\tau-}&{}:=\{\vt\in\Theta\,:\,\vt\one_{\llbracket0,\tau\llbracket}=0\},\qquad{\tau\in\T}\text{ predictable}.
\end{align*}
\end{definition}

As is well known, the set of terminal wealths generated by simple strategies need not be $L^2$-closed; see Monat and Stricker~\cite{monat.stricker.95},  Delbaen, Monat, Schachermayer, Schweizer, and Stricker~\cite{delbaen.al.97}, and Choulli, Krawczyk, and Stricker~\cite{choulli.al.98}. The next definition taken from \cite{czichowsky.schweizer.13} and building up on \cite{cerny.kallsen.07} offers a way out.
\begin{definition}
A trading strategy $\vt\in L(S)$ is called \emph{admissible} if there exists a sequence $(\vt^n)_{n\in\N}$ of simple trading strategies $\vt^n=(\vt^n_t)_{0\leq t\leq T}\in\Theta$, called \emph{approximating sequence for $\vt$}, such that
\begin{enumerate}[(1)]
\item $\vt^n\sint S_T\overset{L^2}{\longrightarrow} \vt\sint S_T$;
\item $\vt^n\sint S_\tau\overset{L^0}{\longrightarrow} \vt\sint S_\tau$ for all $\tau\in\T$.
\end{enumerate}
The set of all admissible trading strategies is denoted by $\TB$, and we further let 
\begin{align*}
\TB_\tau&{}:=\{\vt\in\TB\,:\,\vt\one_{\llbracket0,\tau\rrbracket}=0\},\qquad{\tau\in\T};\\
\TB_{\tau-}&{}:=\{\vt\in\TB\,:\,\vt\one_{\llbracket0,\tau\llbracket}=0\},\qquad{\tau\in\T}\text{ predictable}.
\end{align*}
\end{definition}


\subsection{Law of one price for the price process \texorpdfstring{$S$}{S}}\label{SS:2.3}

In layperson's terms, the law of one price states that portfolios generating the same terminal payoff by trading in the market $S$ should have the same value at all earlier times, thus eliminating the most conspicuous arbitrage opportunities. Then, by the linearity of portfolio formation, all portfolios generating zero payoff at maturity $T$ should have price zero at all earlier times. Since we wish to consider pricing rules that are continuous in the $L^2$ space of payoffs, the law of one price must be true also approximately in $L^2$. This yields the requirement~\ref{def:lop:1} below for trading between arbitrary stopping time $\tau$ and the terminal date $T$. 

For fixed $\tau$, such line of reasoning is very natural and appears, for example, in Hansen and Richard \cite{hansen.richard.87} in the context of static trading with infinitely many assets. Observe that in a finite discrete time setup, the continuity requirement is unnecessary since the set of terminal wealth distributions attainable by trading is closed in $L^0$, hence also in $L^2$. 

The crucial step in this paper is a deeper analysis of predictable times. For a predictable time, say $\sigma$, it is possible to start trading immediately before $\sigma$, at time $\sigma-$ so to say. Requirement \ref{def:lop:2} below formulates the law of one price for such trades. Here it becomes important that the approximation that was previously static takes on a temporal dimension over a sequence of announcing times.

\begin{definition}\label{D:LOP1S}
We say that \emph{the price process $S=(S_t)_{0\leq t\leq T}$ satisfies the law of one price} if the following conditions hold.
\begin{enumerate}[(1)]
\item\label{def:lop:1} For all stopping times $\tau\in\T$, all $\F_\tau$-measurable endowments $x_\tau$ and all sequences $(\vt^n)_{n=1}^\infty$ of simple trading strategies such that $x_\tau+\vt^n\one_{\rrbracket\tau,T\rrbracket}\sint S_T\xrightarrow{L^2}0$, we have $x_\tau=0$. 
\item\label{def:lop:2} Let $\sigma\in\T$ be a predictable stopping time and $(\sigma_n)_{n=1}^\infty$ be any announcing sequence of stopping times for $\sigma$. Then, for all sequences of $\F_{\sigma_n}$-measurable endowments $(x^n_{\sigma_n})_{n=1}^\infty$ and $(\vt^n)_{n=1}^\infty$ of simple trading strategies such that 
$x^n_{\sigma_n}+\vt^n\one_{\rrbracket\sigma_n,T\rrbracket}\sint S_T\xrightarrow{L^2}0$ and 
$x^n_{\sigma_n}\xrightarrow{L^0} x_{\sigma-}$ for some random variable $x_{\sigma-}$, we have that $x_{\sigma-}=0$.  \end{enumerate}
\end{definition}
\begin{remark}\label{rem:uniqueness}
The law of one price for $S$ has the following easy consequences.
\begin{itemize}
\item Since the simple strategies approximate admissible strategies in $L^2$ at maturity and in $L^0$ at intermediate times, the law of one price extends to admissible strategies by a diagonal sequence argument. 
\item The wealth process of an admissible strategy is uniquely determined by its terminal value, i.e., if for $\vt$,$\tilde{\vt}\in\TB$ one has $\vt\sint S_T=\tilde\vt\sint S_T$, then $\vt\sint S$ and $\tilde\vt\sint S$ are indistinguishable. 
\item In the setting of item~\ref{def:lop:2}, the $L^0$-limit of the sequence $\vt^n\sint S_{\sigma_n}$, should it exist for a given sequence of strategies $\vt^n\in\TB$, does not depend on the announcing sequence of stopping times $(\sigma_n)_{n=1}^\infty$. 
\end{itemize}
\end{remark}


\subsection{Price systems}
\begin{definition}\label{D:PSlop}
A family $\{p_\tau\}_{\tau\in\T}$ of pricing operators $p_{\tau}:L^2(\F_T,P)\to L^0(\F_\tau,P)$ is a \emph{price system satisfying the law of one price} if the following properties hold for every $\tau\in\T$.
\begin{enumerate}[(1)]
\item[\optionaldesc{(1)}{D:PSlop:1}] {\bf Correct pricing of the riskless asset.} We have $p_\tau(1)=1$.
\item[\optionaldesc{(2)}{D:PSlop:2}] {\bf Time consistency.} We have that $p_\sigma$ satisfies $p_\sigma(p_\tau(H))=p_\sigma(H)$ for each $H\in L^2(\F_T,P)$ such that $p_\tau(H)\in L^2(\F_\tau,P)$.
\item[\optionaldesc{(3a)}{D:PSlop:3a}]{\bf Conditional linearity.} We have $p_\tau(a_1H_1+a_2H_2)=a_1p_\tau(H_1)+a_2p_\tau(H_2)$
for all $H_1,H_2\in L^2(\F_T,P)$ and $a_1,a_2\in L^\infty(\F_\tau,P)$.
\item[\optionaldesc{(3b)}{D:PSlop:3b}] {\bf Conditional continuity.} Let $(H_n)_{n=1}^\infty$ be a sequence of random variables $H_n$ that converges to some $H$ in $L^2(\F_T,P)$. Then, $p_{\tau}(H_n)$ converges to $p_{\tau}(H)$ in $L^0(\F_\tau,P)$.
\item[\optionaldesc{(4)}{D:PSlop:4}] {\bf Left limits at predictable stopping times.} For predictable $\tau\in\T$, any announcing sequence  $(\tau_n)_{n=1}^\infty$ for $\tau$, and any $H\in L^2(\F_T,P)$, we have that $p_{\tau_n}(H)$ converges to some random variable in $L^0(\F_T,P)$. The limiting random variable, independent of the announcing sequence $(\tau_n)_{n=1}^\infty$, will be denoted $p_{\tau-}(H)$.
\end{enumerate}
Furthermore, a family $\{p_\tau\}_{\tau\in\T}$ of pricing operators $p_{\tau}:L^2(\F_T,P)\to L^0(\F_\tau,P)$ is said to be \emph{compatible} with $S$, if 
$p_\tau(S_{\tau_n})=S_{\tau_n\wedge \tau}$ for all $n\in\N$, all $\tau\in\T$, and any localizing sequence $(\tau_n)_{n=1}^\infty$ of stopping times $\tau_n$ converging stationarily to $T$ such that $S^*_{\tau_n}\in L^2(\F_T,P)$.
\end{definition}

Observe that $p_\tau(H)$ is the terminal wealth obtained by investing the time-$\tau$ price of $H$ into the risk-free asset. To uphold the law of one price, the prices of $H$ and $p_\tau(H)$ must therefore be the same at time $\tau$. This indeed follows from conditions \ref{D:PSlop:1} and \ref{D:PSlop:3a} of Definition \ref{D:PSlop}. Time consistency \ref{D:PSlop:2} further asserts that the two portfolios $p_\tau(H)$ and $H$ have the same value at all \emph{earlier} stopping times $\sigma\leq \tau$. This natural LOP  condition does not follow from \ref{D:PSlop:1} and \ref{D:PSlop:3a}.  Conditions \ref{D:PSlop:3b} and \ref{D:PSlop:4} yield time consistency also at left limits of predictable times, i.e., one has $p_{\sigma-} = p_{\sigma-}(p_{\tau-})$ for all predictable $\sigma,\tau\in\T$ with $\sigma\leq\tau$.

Conditions \ref{D:PSlop:3a}, \ref{D:PSlop:3b}, and \ref{D:PSlop:4} reflect, in this order, increasing generality of modelling frameworks. Combined with \ref{D:PSlop:1} and \ref{D:PSlop:2}, conditional linearity \ref{D:PSlop:3a} gives LOP on finite probability spaces (e.g., Koch-Medina and Munari \cite{koch-medina.munari.20})  and also in finite discrete-time models with finitely many assets (e.g., Courtault et al. \cite{courtault.al.04}); finite discrete time with infinitely many assets needs also continuity~\ref{D:PSlop:3b} (e.g., Hansen and Richard \cite{hansen.richard.87}); continuity at predictable times \ref{D:PSlop:4} has not been studied previously.

\begin{remark}
Since $L^2(\F_T,P)$ and $L^0(\F_\tau,P)$ are completely metrizable topological vector spaces and $p_\tau$ is linear, the continuity \ref{D:PSlop:3b} is equivalent to the property that $p_\tau$ has a closed graph in $L^2(\F_T,P)\times L^0(\F_\tau,P)$ by  
\cite[Theorem~5.20]{aliprantis.border.06}. By the linearity of $p_{\tau}$, this is equivalent to the closed graph property at $0$.
\end{remark}


\subsection{State price densities}\label{SS:2.6}
A conditionally linear and continuous pricing operator $p_{\tau}:L^2(\F_T,P)\to L^0(\F_\tau,P)$ can naturally be represented as a conditional expectation involving a conditionally square-integrable random variable. Namely, by a conditional version of the Riesz representation theorem for Hilbert spaces in Hansen and Richard \cite[Theorem~2.1]{hansen.richard.87}, there is an $\F_T$-measurable random variable $\tauZ_T$, commonly known as the \emph{state price density}, such that 
\begin{equation*}
\E[\tauZ_T^2\,|\,\F_\tau]<\infty\quad\text{ and }\quad p_\tau(H) = \E[\tauZ_T H \,|\,\F_\tau] \quad\text{ for all } H\in L^2(\F_T,P),\tau\in\T.
\end{equation*}

Let us rephrase the notion of a `price system satisfying the law of one price' in terms of the corresponding family of state price densities  $\{\tauZ_T\}_{\tau\in\T}$. 
\begin{definition}[State price density satisfying the law of one price]\label{D:SPDlop}
 We say that a family of $\F_T$-measurable random variables  $\{\tauZ_T\}_{\tau \in \T}$ is a \emph{state price density satisfying the law of one price} if the following conditions hold for every stopping time $\tau\in\T$.%
\begin{enumerate}[(1)]
\item\label{spd.1} {\bf Correct pricing of the risk-free asset.} One has $\tauZ_\tau:=\E[\,\tauZ_{T}\,|\,\F_{\tau}]=1$.
\item\label{spd.2} {\bf Time consistency.} For all $\sigma\leq\tau\in\T$, one has $\sigZ_{T}=\sigZ_\tau\tauZ_{T}$
with $\sigZ_\tau:=\E[\sigZ_{T}\,|\,\F_\tau]$.
\item\label{spd.3} {\bf Conditional square integrability.} One has $\E[\,\tauZ_T^2\,|\,\F_\tau]<\infty$.
\item\label{spd.4} {\bf Bounded conditional second moments before predictable stopping times.} For any predictable $\tau$ and any announcing sequence  $(\tau_n)_{n=1}^\infty$ for $\tau$, we have that the $\F_{\tau-}$-measurable random variable $C:=\sup_{n\in\N}\E[{}^{\tau_n}\mkern-1.5mu Z_T^2\,|\,\F_{\tau_n}]$ is finite.
\end{enumerate}
Furthermore, we say that a family of random variables $\{\tauZ_T\}_{\tau\in\T}$ is \emph{compatible} with $S$ if 
$$\E[S_{\tau_n}\mkern1mu\tauZ_T\,|\,\F_\tau]=S_{\tau_n\wedge \tau}$$ for all $n\in\N$, all $\tau\in\T$, and any localizing sequence $(\tau_n)_{n=1}^\infty$ of stopping times $\tau_n$ converging stationarily to $T$ such that $S^*_{\tau_n}\in L^2(\F_T,P)$.
\end{definition}

\begin{remark}
Process $\tauZ=(\tauZ_t)_{0\leq t\leq T}$ that arises in items \ref{spd.1} and \ref{spd.2} by setting $\tauZ_t=\E[\tauZ_T\,|\,\F_t]$ for $0\leq t\leq T$ is  not necessarily a uniformly integrable martingale. However,  for $K:=\E[|\tauZ_T|\,|\,\F_\tau]\vee 1$ the adapted process $\frac{\tauZ}{K}\one_{\rrbracket \tau,T\rrbracket}$ coincides on $\rrbracket \tau,T\rrbracket$ with the uniformly integrable martingale closed by $\frac{\tauZ_T}{K}$. This also shows that $\tauZ$ is a semimartingale.
\end{remark}


\subsection{\texorpdfstring{$\Exp$}{E}--densities and \texorpdfstring{$\Exp$}{E}--martingales}\label{SS:2.7}

We next recall and slightly adapt the concept of $\Exp$-martingales introduced by Choulli, Krawczyk, and Stricker \cite{choulli.al.98}. 
For any stopping time $\tau$, we denote the process $Y$ stopped at $\tau$ by $Y^\tau$ and the process $Y$ restarted at $\tau$ from 0 by ${}^\tau Y=Y-Y^\tau$; but we set $\tauE(N)=\Exp(N-N^\tau)$. The stochastic exponential $\tauE(N)$ therefore denotes a multiplicative restarting from 1 rather than an additive restarting from 0. We shall now study the family of processes $\{\tauE(N)\}_{\tau\in\T}$.
\begin{definition}\label{D:Edens}
We say that the family $\{\tauE(N)\}_{\tau\in\T}$ is an $\Exp$-density if $\E[\,\tauE(N)_T\,|\,\F_\tau]=1$ for all 
$\tau\in\T$.  An $\Exp$-density $\{\tauE(N)\}_{\tau\in\T}$ will be called \emph{square integrable} if one has $\E[\,\tauE(N)_T^2\,|\,\F_{\tau}]<\infty$ for all $\tau\in\T$.  
\end{definition}

\begin{definition}\label{D:Emart}
An adapted RCLL process $Y$ is an \emph{$\Exp(N)$-martingale} if 
$$ Y_\tau = \E[\,\tauE(N)_TY_T\,|\,\F_\tau],\qquad \tau\in\T.$$
We say $Y$ is an \emph{$\Exp(N)$-local martingale} if there is a localizing sequence of stopping times $\tau_n$ such that $Y^{\tau_n}$ is an $\Exp(N)$-martingale for each $n\in\N$.  
\end{definition}
\begin{remark}\label{R:Tm}
We have slightly generalised the original definition of $\Exp$-martingales (see Definition 3.11 in~\cite{choulli.al.98}) by imposing milder integrability conditions.  To see this, observe that by \cite[p.~884]{choulli.al.98} for each semimartingale $N$ there is a sequence of stopping times increasing to $\infty$ such that $T_0=0$ and $T_{m+1}=\Inf\{t>T_m\, |\, \TmE(N)_t=0\}$.  The following are then equivalent.
\begin{enumerate}[(i)]
\item $Y$ is an $\Exp(N)$-martingale in the sense of Definition~\ref{D:Emart}.
\item There is a sequence of $\F_{T_m}$-measurable positive random variables $K_m$ such that 
$$
\E\big[|Y_{T_m}K_m\,\TmE_{T_{m+1}}|\big]<\infty
$$
and ${}^{T_m}\!Y K_m\,\TmE$ is a $P$-martingale for all $m\in\N\cup\{0\}$.
\end{enumerate}
In \cite{choulli.al.98} the random variables $(K_m)_{m=0}^\infty$ are not needed since the combination of the assumptions that $Y$ is square integrable and $\{\tauE(N)\}_{\tau\in\T}$ satisfies a reverse H\"older inequality permits taking $K_m\equiv1$ for all $m\in\N\cup\{0\}$.  
\end{remark}

The next two propositions, which mirror \cite[Proposition~3.15, Corollary 3.16 and 3.17]{choulli.al.98}, are reminiscent of the Girsanov theorem for absolutely continuous measure changes; the original proofs in \cite{choulli.al.98} still work for our generalizations.
\begin{proposition}\label{P:elocalmart}
For a semimartingale $Y$ and an $\Exp$-density $\Exp(N)$, the following are equivalent.
\begin{enumerate}[(i)]
\item\label{elocalmart.i} $Y$ is an $\Exp(N)$-local martingale.
\item\label{elocalmart.ii} $Y+[Y,N]$ is a $P$-local martingale. 
\end{enumerate}
Furthermore, if either of the conditions holds and $Y$ is special (with the local martingale part $M^Y$), then 
$$  Y = Y_0 + M^Y -\langle M^Y,N\rangle.$$
\end{proposition}
\begin{proposition}\label{P:emart}
Assume $\Exp(N)$ is an $\Exp$-density and $Y$ is an $\Exp(N)$-local martingale. Consider the sequence of stopping time $(T_m)_{m=0}^\infty$ from Remark~\ref{R:Tm}. If there is a sequence of positive $\F_{T_m}$-measurable random variables $(K_m)_{m=0}^\infty$ such that $\E\big[K_mY^*_T\,(\TmE(N))^*_T\big]<\infty$ for all $m\in\N\cup\{0\}$, then $Y$ is an $\Exp(N)$-martingale.
\end{proposition}

To assist the reader, we conclude this subsection by linking $\Exp$-martingales and $\Exp$-densities to equivalent measures and their densities.
\begin{remark}\label{R:231215}
For a positive $\Exp$-density $\Exp(N)$, the following are equivalent. 
\begin{enumerate}[(i)]
\item $Y$ is a $Q$-martingale for the equivalent measure $Q$ given by $\frac{dQ}{dP}=\Exp(N)_T$.
\item $Y$ is an $\Exp(N)$-martingale and $Y_0$ is integrable.
\end{enumerate}
Thus for an $\Exp$-density $\Exp(N)>0$, an $\Exp(N)$-martingale coincides with the notion of `generalised $Q$-martingale' in Dellacherie and Meyer \cite[p.~3]{dellacherie.meyer.82}.  
\end{remark}


\section{Main results and counterexamples}\label{S:3}


\subsection{Equivalent characterizations of LOP}\label{SS:3.1}
We first recall an important concept from the quadratic hedging literature.
\begin{definition}\label{D:L}
The process $L=(L_t)_{0\leq t\leq T}$ given by
\begin{equation}\label{L}
L_t:=\essinf_{\vt\in\Theta_t}\E\left[(1-\vt\sint S_T)^2\,|\,\F_t\right],\qquad 0\leq t\leq T,
\end{equation}
is called the \emph{opportunity process}.  
\end{definition}

\begin{theorem}\label{T:M1}
For a locally square-integrable semimartingale $S$, the following are equivalent.
\begin{enumerate}[(i)]
\item\label{MT1.1} The law of one price holds for the price process $S$ (Definition~\ref{D:LOP1S}).
\item\label{MT1.2} $S$ admits a compatible price system satisfying LOP (Definition~\ref{D:PSlop}). 
\item\label{MT1.3} $S$ admits a compatible state price density satisfying LOP (Definition~\ref{D:SPDlop}). 
\item\label{MT1.4} There exists a semimartingale $N$ such that $S$ is an $\Exp(N)$-local martingale and $\Exp(N)$ is a square-integrable $\Exp$-density (Definitions~\ref{D:Edens} and \ref{D:Emart}).
\item\label{MT1.5} The opportunity process $L$ and its left limit $L_{-}$ are positive (Definition~\ref{D:L}).
\end{enumerate}
Furthermore, if any of the above conditions holds, then 
\begin{enumerate}[resume*]
\item\label{MT1.6} The subspace $\{\vt\sint S_T:\vt\in\TB\}$ is closed in $L^2$. Hence, a unique (up to a null strategy) solution of \eqref{MVH1} exists in $\TB$.
\item\label{MT1.7} For every $\tau\in\T$, the set $\{\vt\sint S_T:\vt\in\TBtau\}$ is closed in $L^2$.
\item\label{MT1.8} For every predictable stopping time $\sigma\in\T$ with announcing sequence $(\sigma_n)_{n=1}^\infty$ one has
\begin{equation*}
\{\vt\sint S_T\,:\,\vt\in\TB_{\sigma-}\} = \cap_{n=1}^\infty\{\vt\sint S_T\,:\,\vt\in\TB_{\sigma_n}\}.
\end{equation*}
\end{enumerate}
\end{theorem}
 In practice, the criterion \ref{MT1.4} is easily verified if there is an equivalent martingale measure with square-integrable density (Remark~\ref{R:231215}). Failing that, the easiest criterion to check is the non-negativity of the opportunity process and its left limit \ref{MT1.5}. Theorem~\ref{vt} below simplifies the verification procedure further in concrete models by dispensing with the need to find the actual opportunity process in favour of an easier task of identifying just a candidate opportunity process. 

Apart from the novel definition of LOP and its link to $\Exp$-martingales, the key improvement in Theorem~\ref{T:M1} compared to Czichowsky and Schweizer \cite[Theorem~6.2]{czichowsky.schweizer.13} is that \emph{a priori} one does not need  the solutions  of  
\begin{equation}\label{A}
\min_{W\in\,\cl\mkern2mu\{\vt\sint S_T\,:\,\vt\in\Theta_\tau\}}\E\big[ (1-W)^2\,|\,\F_\tau \big],\qquad \tau\in\T,
\end{equation}
to be realized by trading strategies in $\TBtau$. Instead, with a further argument, one obtains that the set $\{\vt\sint S_T:\vt\in\TBtau\}$ is closed \emph{a posteriori}, purely on the strength of any of the items \ref{MT1.1} to \ref{MT1.5}. Observe that the conditions \ref{MT1.1}--\ref{MT1.5} of Theorem \ref{T:M1} are not necessary for the $L^2$-closedness of $\{\vt\sint S_T:\vt\in\TBtau\}$. This can be seen in finite discrete time from the results of Melnikov and Nechaev \cite{melnikov.nechaev.99}. In continuous time, a further counterexample appears in Delbaen et al. \cite[Example~6.4]{delbaen.al.97}. In  this example, $\{\vt\sint S_T:\vt\in\TB\}=L^2$ and $\F_0$ is trivial. Therefore, $1\in\{\vt\sint S_T:\vt\in\TB\}$ and hence $L_0=0$.
Thus, if \ref{MT1.1}--\ref{MT1.5} of Theorem \ref{T:M1} fail, no firm conclusions on the closedness of $\{\vt\sint S_T:\vt\in\TBtau\}$ can be drawn.

\begin{remark}[Variance-optimal state price density]\label{rem:vo}
Because the law of one price for wealth transfers between 0 and $T$ does not imply LOP on subintervals (e.g., $L_0>0$ does not automatically yield $L_t>0$ for $t>0$), the proof of Theorem \ref{T:M1} is not based, unlike other variants of ``Fundamental Theorems of Asset Pricing'',  on an application of a separating hyperplane theorem. Rather, the proof constructs via Lemma~\ref{lem:hit0} a specific family of state price densities $\{\tauhZ_T\}_{\tau\in\T}$ whose elements $\tauhZ_T$ are characterised by having the smallest conditional second moment. Because one has $\E[\tauZ_T\,|\,\F_\tau]=1$ for all $\tau\in\T$ and all families $\{\tauZ_T\}_{\tau\in\T}$ of state price densities by Definition~\ref{D:SPDlop}\ref{spd.1}, the elements of the minimal family  $\{\tauhZ_T\}_{\tau\in\T}$ then also have the smallest conditional variance $\Var(\tauhZ_T\,|\,\F_\tau)$ among all compatible state price densities satisfying LOP.  The family $\{\tauhZ_T\}_{\tau\in\T}$ is in this sense \emph{variance optimal}; cf. Schweizer~\cite[p.~210]{schweizer.96}. In the same vein, for $N$ such that $\tauhZ_T=\tauE(N)_T$, $\tau\in\T$, one may speak of the \emph{variance-optimal $\Exp$-density}.  
\end{remark}
It is shown in Theorem \ref{T:M2}\ref{T:main.2} and \eqref{eq:tauhZ} below that, under LOP for $S$, the variance-optimal $\Exp$-density has the form
$$\tauhZ=\frac{\tauE(-a\sint S)L}{L^\tau}=\tauE(-a\sint S+\Log(L)-[a\sint S,\Log(L)])$$
for some $a\in L(S)$. For continuous $S$, this yields ${}^0\!\hat{Z}>0$ and hence that $S$ admits an equivalent local martingale measure with square-integrable density. Thus, for continuous $S$, LOP is equivalent to the condition of \emph{no `$L^2$ free lunch'} by Theorems 2 and 3 in Stricker \cite{stricker.90}.
\begin{proposition}\label{prop:VOMM}
For a continuous price process $S$, the following are equivalent.
\begin{enumerate}[(i)]
\item\label{VOMM.i} $S$ satisfies the law of one price.
\item\label{VOMM.ii} $S$ admits an equivalent local martingale measure with square-integrable density.
\item\label{VOMM.iii} The variance-optimal (signed) local martingale measure for $S$ exists and is positive. 
\item\label{VOMM.iv} There is no `$L^2$ free lunch', i.e.,
$$\cl\big(\{\vt\sint S_T~:~\vt\in\Theta\}-L^2_+\big)\cap L^2_+=\{0\}.$$
\item \label{VOMM.v} There is no arbitrage on the $L^2$ closure of simple strategies, i.e.,
$$\cl\{\vt\sint S_T~:~\vt\in\Theta\}\cap L^2_+=\{0\}.$$
\end{enumerate}
\end{proposition}
This recovers and extends the celebrated result of Delbaen and Schachermayer \cite{delbaen.schachermayer.96.bej}, i.e., the equivalence of \ref{VOMM.ii} and \ref{VOMM.iii}. Observe that in \cite{delbaen.schachermayer.96.bej} \ref{VOMM.ii} is assumed, while here \ref{VOMM.ii} and \ref{VOMM.iii} follow from the weaker LOP assumption \ref{VOMM.i}.


\subsection{Counterexample}\label{SS:3.2}
The following example illustrates various phenomena that arise when $L_->0$,  and hence also the law of one price, fails. It is striking that the example operates with a \emph{continuous} price process.
\begin{enumerate}[(1)]
\item\label{ex:1} For $\sigma=\inf\{t>0\,|\,L_{t-}=0\}$, the event $F:=\{L_{\sigma-}=0\}\in\F_{T-}$ occurs with positive probability, while $L_t>0$ for all $t\in[0,T]$.
\item\label{ex:2} There is a state price density $\{\tauhZ_T\}_{\tau\in\T}$ compatible with $S$ that satisfies properties \ref{spd.1}--\ref{spd.3} of Definition~\ref{D:SPDlop}. Furthermore, this family of random variables can be chosen such that $\E[(\tauhZ_T)^2\,|\,\F_\tau]=\frac{1}{L_\tau}$ for all $\tau\in\T$. However, $\{\tauhZ_T\}_{\tau\in\T}$ does not have bounded conditional second moments before predictable stopping times, that is, property \ref{spd.4} of Definition~\ref{D:SPDlop} fails.

\item\label{ex:3} The price process satisfies the no-arbitrage (NA) condition on $[0,T]$ for $L^\infty$-admissible as well as for $L^2$-admissible strategies, that is,
$$\text{$\{\vt\sint S_T\,:\,\vt\in\Theta^\infty\}\cap L^0_+=\{0\}$ and $\{\vt\sint S_T\,:\,\vt\in\TB\}\cap L^2_+=\{0\}$},$$
where $\Theta^\infty=\{\vt\in L(S) \,:\,\inf_{t\in[0,T]}\vt\sint S_t\in L^\infty\}$.

\item\label{ex:4} The subspace $\{\vt\sint S_T\,:\,\vt\in\TB\}$ fails to be closed in $L^2$, even though $L_t>0$ for all $t\in[0,T]$.
\item\label{ex:5} There is an announcing sequence $(\sigma_n)_{n=1}^\infty$ to the predictable stopping time $\sigma$ in \ref{ex:1} such that 
\begin{equation*}
\{\vt\sint S_T\,:\,\vt\in\TB_{\sigma-}\}\ne\cap_{n=1}^\infty\mkern2mu\cl\mkern2mu\{\vt\sint S_T\,:\,\vt\in\TB_{\sigma_n}\}.
\end{equation*} 
\item\label{ex:6} The price process admits an absolutely continuous local martingale measure (ACLMM) $Q$ with square-integrable density, but no equivalent martingale measure. In particular, $Q(F)=0$ for the non-null event $F\in\F_{T-}$ defined in \ref{ex:1}. Indeed, $S$ starts at $1$ and terminates at $0$ on $F$.

\item\label{ex:7} In contrast to \ref{ex:3}, there is a `free lunch with vanishing risk' --- a sequence of trading strategies with wealth bounded below by $-\frac{1}{n}$ such that the wealth of each strategy is $1$ on $F$. Likewise, there is an `$L^2$ free lunch' --- a sequence of simple zero-cost strategies which after disposal of an $L^2$-integrable non-negative amount converges in $L^2$ to a non-zero element of $L^2_+$, i.e.,
$$\cl\big(\{\vt\sint S_T~:~\vt\in\Theta\}-L^2_+\big)\cap L^2_+\neq\{0\}.$$  
\end{enumerate} 
\begin{example}\label{example} Let $W$ be a Brownian motion in its natural filtration. For $T:=1$ and $t\in[0,T]$, we set $X_t=(T-t)\Exp(W)_t$. Let $\tau$ be an independent stopping time such that $P(\tau=T)=p\in(0,1)$ and $\tau$ is uniformly distributed on $[0,T)$ with probability $1-p$. Define the stock price by $S = X^\tau$. Then, 
$$\frac{\d S_t}{S_t} = \mu_t \d t + \d W_t,\quad t\in[0,T),$$
where $\mu_t=-\frac{1}{T-t}\one_{\llbracket0, \tau\llbracket}$ and we used that 
$T-t=\e^{\log(T-t)}=\e^{-\int_0^t\frac{1}{T-s}\d s}$ for $t\in[0,T)$.

Let us highlight the key points in the construction of the example. The continuous process $X$ starts at $1$, is positive on $[0,T)$, and equals $0$ at $T$. It is constructed to admit an equivalent local martingale measure with square-integrable density on each closed subinterval of $[0,T)$. Hence, trading on $X$ fails the (NA) condition on $[0,T]$ but satisfies it on $[0,s]$ for every $s<T$.  The stopping time $\tau$ is chosen to satisfy $P(t<\tau<T \,|\, \F_t)>0$ for every $t\in[0,T)$, which yields that the stopped process $S=X^\tau$ satisfies (NA) on the whole time interval $[0,T]$.  
Intuitively, in order to realize an arbitrage opportunity, one must start trading at some stopping time $\varrho$ with $P(\varrho<\tau)>0$. However, because $\tau$ is totally inaccessible on $[0,T)$, such trade is active also on the smaller non-null event $\{\varrho<\tau<T\}$, where the trading gains take both signs as $X$ is independent of $\tau$. A rigorous proof is supplied in item \ref{ex:3}.

The situation changes as soon as one considers the concept of a `free lunch with vanishing risk' (FLVR). While an (NA) violation  is realized by a single strategy, FLVR allows one to choose an entire sequence of strategies that get increasingly closer to an arbitrage opportunity. In this example, it is significant that  an arbitrarily small initial capital can be turned into $1$ at maturity on the event $\{\tau=T\}$ by an admissible strategy whose wealth never drops below zero. The FLVR strategy borrows $\frac{1}{n}$ at the risk-free rate at time zero and places this amount into a closed-end fund whose policy is to have proportion $-\frac{1}{T-t}\one_{\llbracket0, \tau\llbracket}$ invested in the risky asset $S$. The increasingly larger short position in the risky asset exploits the fact that $S$ is drifting strongly towards zero  on the predictable set $\{\tau=T\}$ as $t\to T$ while the conditional probability $P(\tau=T~|~\tau>t)$ increases to $1$. The fund is liquidated if it ever reaches the value of $1$, which is guaranteed to happen on the event $\{\tau=T\}$. We now dispose of the fund value built up on the complement $\{\tau<T\}$. After repaying the risk-free borrowing, this yields the FLVR sequence of terminal wealths $(\one_{\{\tau=T\}} - \frac{1}{n})_{n\in\N}$. The strategy earns $1-\frac{1}{n}$ with probability $P(\tau=T)\equiv 1-p >0$ and loses no more than $\frac{1}{n}$ with probability of no more than $p>0$. The details are found in item \ref{ex:7}.

\ref{ex:1}: Let $\vp^n=-\Exp\big(-(\frac{\mu}{S}\one_{\llbracket 0,T-\frac{1}{n}\rrbracket})\sint S\big)\frac{\mu}{S}\one_{\llbracket 0,T-\frac{1}{n}\rrbracket}$ so that
$$1+\vp^n\sint S_T=\Exp\left(-\left(\frac{\mu}{S}\one_{\llbracket 0,T-\frac{1}{n}\rrbracket}\right)\sint S\right)_T=\Exp\left(-\frac{\mu}{S}\sint S\right)_{T-\frac{1}{n}}.$$
 Observe that $\vp^n\sint S$ is an $\mathcal{H}^2$ semimartingale (Protter \cite[Chapter IV]{protter.05}) and therefore can be approximated by stochastic integrals of simple strategies in $\mathcal{H}^2$ so that each $\vp^n\in\TB$ by \cite[Theorem IV.2]{protter.05}. 
The opportunity process $L=(L_t)_{0\leq t\leq T}$ is given by
\begin{align*}
L_t&{}=\lim_{n\to\infty}\E\left[ {}^t\!\Exp\big(-\frac{\mu}{S}\sint S\big)_{T-\frac{1}{n}}^2~\big|~\F_t\right]
=\lim_{n\to\infty}\E\left[\e^{-\int_{\tau\wedge t}^{\tau\wedge (T-1/n)} \mu_s^2\d s}~\big|~\F_t\right]\\
&{}=\E\left[\e^{-\int_{\tau\wedge t}^\tau\mu^2_s\d s}\one_{\{\tau<T\}}~\big|~\F_t\right]
=\one_{\{\tau\leq t\}}+\one_{\{\tau > t\}}(1-p)\e^{(T-t)^{-1}}\int_{[t,T)}  \e^{-(T-u)^{-1}}\d u\\
&{}\leq\one_{\{\tau\leq t\}}+\one_{\{\tau > t\}}(1-p)(T-t),\qquad t\in[0,T],
\end{align*}
yielding $L_t>0$ for all $t\in[0,T]$ and $L_{T-}=\lim_{t\uparrow T}L_t=0$ on $\{\tau=T\}$ with $P(\tau=T)=p>0$.

\ref{ex:2}: Note that $\frac{\mu}{S}\one_{\llbracket0,T-\frac{1}{n}\rrbracket}$ is the standard adjustment process on $[0,T-\frac{1}{n}]$ for any $n\in\N$
but $\frac{\mu}{S}$ itself is not in $L(S)$. Nonetheless, since $\Exp\big(-\frac{\mu}{S}\sint S\big)_{T-\frac{1}{n}}\to 0$ on $\{\tau=T\}$, one can define 
$\vp:=-\frac{\mu}{S}\Exp(-\frac{\mu}{S}\sint S)\one_{\llbracket0,T\llbracket}=\lim_{n\to\infty}\vp^n$, which gives an integrand in $L(S)$ such that $1+\vp\sint S_T=\lim_{n\to\infty}\Exp\big(-\frac{\mu}{S}\sint S\big)_{T-\frac{1}{n}}$. See also Liptser and Shiryaev \cite[Subsubsection~6.1.4]{liptser.shiryaev.01.part1} for details about stochastic exponentials of stochastic integrals of Brownian motion hitting zero. Because of the independence of $W$ and $\tau$, we have $\E[(1+\vp\sint S_T)^2]=\E[\exp(-\int_{0}^{\tau} \mu_s^2\d s)]=\E[\exp(-\int_{0}^{\tau} \mu_s^2\d s)\one_{\{\tau<T\}}]=L_0.$ This yields $1+\vp^n\sint S_T= \Exp\big(-\frac{\mu}{S}\sint S\big)_{T-\frac{1}{n}}\xrightarrow{L^2}1+\vp\sint S_T$ and therefore $\vp\in\TB$ by approximating it with a diagonal sequence of simple strategies. Property \ref{ex:2} follows directly from Lemma \ref{lem:hit0} below and the fact that $P(L_{\sigma-}=0)=P(\tau=T)=p>0$ by \ref{ex:1}, because $L>0$. This also yields that $Q$ defined via
\begin{equation}\label{ex:ACMM}
\frac{dQ}{dP}=\frac{1+\vp\sint S_T}{\E[1+\vp\sint S_T]}={}^0\mkern-1mu\hat{Z}_T\geq 0
\end{equation}
is an ACLMM for $S$ with square-integrable density ${}^0\mkern-1mu\hat{Z}=({}^0\mkern-1mu\hat{Z}_t)_{0\leq t\leq T}$ and
$$\left\{\frac{dQ}{dP}=0\right\}=\{L_{\sigma-}=0\}=\{\tau=T\}.$$

\ref{ex:3}: We begin by showing the $L^\infty$-(NA) property, that is, $\{\vt\sint S_T\,:\,\vt\in\Theta^\infty\}\cap L^0_+=\{0\}$. For a proof by contradiction, suppose that the $L^\infty$-(NA) property fails, that is, there is $\psi\in\Theta^\infty$ such that $\psi\sint S_T\in L^0_+\setminus\{0\}$. Denote by ${}^0\mkern-1mu\hat{Z}$ the square-integrable density process of the ACLMM $Q$ from \ref{ex:2}. Since $\psi\sint S$ has a uniform lower bound, the local martingale ${}^0\mkern-1mu\hat{Z}(\psi\sint S)$ is a supermartingale. Non-negativity of ${}^0\mkern-1mu\hat{Z}_T$ and $\psi\sint S_T$ together with the supermartingale property yield ${}^0\mkern-1mu\hat{Z}(\psi\sint S)=0$ on $[0,T]$. 
This in turn gives $\psi\sint S_t=0$ for all $t\in[0,T)$ since ${}^0\mkern-1mu\hat{Z}>0$ on $[0,T)$. By the continuity of $S$, one has $\psi\sint S_T=0$, which contradicts $P(\psi\sint S_T>0)>0$. 

The proof of the $L^2$-(NA) property, that is, $\{\vt\sint S_T\,:\,\vt\in\TB\}\cap L^2_+=\{0\}$, proceeds similarly. Indeed, suppose again, for a proof by contradiction, that the $L^2$-(NA) property fails and there is $\psi\in\TB$ such that $\psi\sint S_T=f\in L^0_+\setminus\{0\}$. Then, ${}^0\mkern-1mu\hat{Z}_t(\psi\sint S_t)$ is a martingale by Lemma \ref{lem:hit0} below and hence $\E[{}^0\mkern-1mu\hat{Z}_T(\psi\sint S_T)]= 0$. As before, the latter contradicts the assumption that $P(\psi\sint S_T>0)>0$.

\ref{ex:4}--\ref{ex:5}: 
Recall that $\varphi = -\frac{\mu}{S}\Exp(-\frac{\mu}{S}\sint S)\one_{\llbracket0,T\llbracket}$ is in $\TB$ and that 
$$ 1+\varphi\sint S = \Exp\left(-\frac{\mu}{S}\sint S\right)^\tau\one_{\{\tau<T\}}.$$
Likewise, $\vartheta^n := -\frac{\mu}{S}\one_{\{T-\frac{1}{n}<\tau\}}\, {}^{T-\frac{1}{n}}\Exp(-\frac{\mu}{S}\sint S)\one_{\llbracket0,T\llbracket}$ is in $\TB$ and one has
$$ \one_{\{T-\frac{1}{n}<\tau\}}+\vartheta^n\sint S = \,{}^{T-\frac{1}{n}}\Exp\left(-\frac{\mu}{S}\sint S\right)^\tau\one_{\{T-\frac{1}{n}<\tau<T\}}.$$ 
Observe that $\vartheta_n$ starts trading at 
$$\sigma_n := \left(T-\frac{1}{n}\right)\one_{\{\tau> T-\frac{1}{n}\}}+T\one_{\{\tau\leq(T-\frac{1}{n})\}},$$
which is an announcing sequence for the predictable stopping time $\sigma := T$. Therefore, the above yields 
$$\E\left[\left(\one_{\{T-\frac{1}{n}<\tau\}}+\vt^n\sint S_T\right)^2\right]=\E\left[\exp\left(-\int_{T-\frac{1}{n}}^{\tau} \mu_s^2\d s\right)\one_{\{T-\frac{1}{n}<\tau<T\}}\right]\to 0$$ 
so that $-\vt^n\sint S_T\xrightarrow{L^2} \one_{\{\tau=T\}}$. Thus $\one_{\{\tau=T\}}\in\cl\,\{\vartheta\sint S_T:\vartheta\in\TB\}$ but there is no $\psi\in\TB$ such that $\psi\sint S_T=\one_{\{\tau=T\}}$ by \ref{ex:3}, because such a $\psi\in\TB$ would violate the $L^2$-(NA) property, which gives \ref{ex:4}. 

Furthermore, because $S$ is continuous and $\sigma=T$, we have $\{\vartheta\sint S_T : \vartheta \in\TB_{\sigma_-}\}=\{0\}$ while 
$$\one_{\{\tau=T\}}\in \cap_{n\in\N}\mkern2mu\cl\mkern2mu\{\vartheta\sint S_T:\vartheta\in\TB_{\sigma_n}\},$$ 
which gives \ref{ex:5}.

\ref{ex:6}: The ACLMM for $S$ has been constructed in part \ref{ex:2}, where it is also shown that $Q(F)=0$. On the other hand, by the fundamental theorem of asset pricing as, for example, in Delbaen and Schachermayer \cite[Theorem~1.1]{delbaen.schachermayer.94}, there cannot be an ELMM for $S$. 

\ref{ex:7}: Observe that the non-negative local martingale $\Exp(-\mu\sint W)\equiv 1/\Exp(\frac{\mu}{S}\sint S)$ on $[0,T)$ has a finite left limit at $T$ and that this limit is $0$ on $\tau=T$. For the stopping times 
$$\sigma_n:=\inf\left\{t>0~:~\Exp\left(\frac{\mu}{S}\sint S\right)_t>n\right\},$$ 
we thus have that $\sigma_n<T$ on the event $\{\tau=T\}$ and that $ \frac{1}{n}(\Exp(\frac{\mu}{S}\sint S)^{\sigma_n}-1)$
is the wealth of an $L^\infty$-admissible trading strategy
$$ \psi^n := \frac{1}{n}\frac{\mu}{S}\Exp\left(\frac{\mu}{S}\sint S\right)\one_{\llbracket0, \sigma_n\rrbracket}\in\Theta^\infty,\qquad n\in\N.$$
Since 
\begin{align*}
\psi^n\sint S_T = \frac{1}{n}\left(\Exp\left(\frac{\mu}{S}\sint S\right)_T^{\sigma_n}-1\right)\geq {}& 
\frac{1}{n}\left(\Exp\left(\frac{\mu}{S}\sint S\right)_T^{\sigma_n}\one_{\{\tau=T\}}-1\right)\\
={}&\one_{\{\tau=T\}}-\frac{1}{n}\xrightarrow{L^\infty} \one_{\{\tau=T\}}
\end{align*}
and $P(\tau=T)>0$, this sequence yields a `free lunch with vanishing risk.'

In the very specific setting of this example, one could modify the sequence $(\psi^n)_{n\in\N}$ by starting to trade at $T-\frac{1}{n}$ solely on the event $\{\tau>T-\frac{1}{n}\}$, which yields a sequence of terminal wealths, where the gain is still $(1-\frac{1}{n})\one_{\{\tau=T\}}$ and worst loss is still $\frac{1}{n}$ but the probability of loss drops from at most $P(\tau<T)$ to at most $\frac{P(\tau<T)}{n}$. 

It remains to argue that the strategies $\psi^n$ are not only in $\Theta^\infty$ but also in $\TB$. To this end, we recall that the wealth processes $\psi^n\sint S$ is valued in $[-1/n,1]$ and hence bounded. Therefore, the integral $\psi^n\sint S$ is a locally square-integrable semimartingale so that there exits a sequence of stopping times $(\varrho_m)_{m=1}^\infty$ such that $\psi^{n,m}:=\psi^n\mathbbm{1}_{\llbracket0,\varrho_m\rrbracket}\in\Theta$ by approximating it with a diagonal sequence of simple strategies from the construction of the stochastic integral. Because $\psi^{n,m}\sint S$ is valued in $[-1/n,1]$ for all $m\in\N$, we have that $\psi^{n,m}\sint S_T\to \psi^n\sint S_T$ in $L^2$ and $\psi^{n,m}\sint S_\varrho\to \psi^n\sint S_\varrho$ in $L^0$, as $m\to\infty$, for any $[0,T]$-valued stopping time $\varrho$ which gives $\psi^n\in\TB$ for all $n\in\N$.
\end{example}

\subsection{Law of one price and wealth transfers}\label{SS:3.3bis}
Let us offer some intuition for the link between the opportunity process $L$ and the law of one price for $S$ announced in Theorem~\ref{T:M1}. Recall that $\cG=\{\vt\sint S_T:\vt\in\Theta\}$ contains the terminal wealth distributions attainable by simple trading with zero initial wealth. The statement 
\begin{equation}\label{LP 0 T}
\text{ 
``the affine subspaces $v+\cG$ are disjoint for different values of $v\in\R$''}
\end{equation}
can be seen as the law of one price for simple wealth transfers between time $0$ and time $T$. Indeed, the terminal wealths in $v+\cG$ are obtainable at the initial price $v$. If the \emph{same} wealth is obtainable at two distinct initial prices, the law of one price no longer applies. Observe that \eqref{LP 0 T} can be restated more compactly as $1\notin \cG$. Since $\cG$ is not necessarily closed, one must strengthen this requirement to 
\begin{equation}\label{1 not in cl G}
 1\notin \cl\cG;
\end{equation}
condition \eqref{1 not in cl G} fails if there exists a fixed random variable in $L^2$ that can be approximated arbitrarily well by elements of $v+\cG$ for two different values of $v\in\R$.

Let us generalise these observations to trading between an arbitrary stopping time $\tau\in\T$ and $T$. To this end, 
let $\cGtau$ contain terminal wealths of simple zero-cost strategies that do not trade on the interval $\llbracket0,\tau\rrbracket$, i.e., 
\[\cGtau:=\{\vt\sint S_T\,:\,\vt\in\Theta_\tau\},\qquad{\tau\in\T}.\]
The appropriate condition now reads
\begin{equation}\label{1A not in cl Gtau}
 \one_A\notin \cl\cGtau\qquad\text{ for all $A\in \F_\tau$ with $P(A)>0$}.
\end{equation}
In analogy with \eqref{1 not in cl G}, the requirement~\eqref{1A not in cl Gtau} may be interpreted as the \emph{law of one price for wealth transfers between times $\tau$ and $T$}: if the condition~\eqref{1A not in cl Gtau} fails for some non-null event $A\in\F_\tau$, then there are strategies with different (constant on $A$) initial wealth at time $\tau$ that approximate the same terminal wealth supported on $A$.

Given this interpretation, it is not difficult to see (on an intuitive level at least) that \eqref{1A not in cl Gtau} corresponds to the condition \ref{def:lop:1} of Definition~\ref{D:LOP1S} and that it yields $L>0$ via \eqref{L}. Theorem~\ref{T:M1} further shows that Definition~\ref{D:LOP1S}\ref{def:lop:2} corresponds to the requirement that for every predictable $\tau\in\T$ and some (equivalently every) announcing sequence $(\tau_n)$ for $\tau$, one has 
\begin{equation}\label{1A not in cap cl Gtau_n}
\one_A\notin \cap_{n}\cl\cG_{\tau_n},\qquad \text{for all }A\in \F_{\tau-},\ P(A)>0,
\end{equation}
which then yields 
\begin{equation}\label{Gtau- right}
\cap_{n=1}^\infty\GB_{\tau_n}=\{\vt\sint S_T\,:\,\vt\in\TB_{\tau-}\}=:\GB_{\tau-},
\end{equation}
and hence $L_->0$. The novel condition~\eqref{1A not in cap cl Gtau_n} may be interpreted as the law of one price for wealth transfers between $\tau-$ and $T$ for predictable $\tau\in\T$. Without the LOP condition \eqref{1A not in cap cl Gtau_n}, the first equality in \eqref{Gtau- right} may fail and this can occur even for \emph{continuous} price processes. Both conditions $L>0$ and $L_->0$ can therefore be seen as dynamic versions of Schweizer's \cite[p. 24]{schweizer.01} requirement
$$\text{\emph{no approximate profits in $L^2$ for $\mathcal{G}$}},$$
which yields the absence of a restricted set of `$L^2$ free lunches' tailored to quadratic optimisation criteria.

With these observations in mind, let us now revisit the notion of a state price density, i.e., the family of conditionally square-integrable random variables $\{\tauZ_T\}_{\tau\in\T}$ that meet some additional internal consistency criteria, such as $\E[\tauZ_T\,|\,\F_\tau]=1$ for all $\tau\in\T$. For a fixed $\tau\in\T$ this family, too, defines a set of  terminal wealths available at zero cost at time $\tau\in\T$, namely
\[\GH_\tau := \{W\in L^2(P\,|\,{\F_\tau}) : \E[W\,\tauZ_T\,|\,\F_\tau] = 0\},\]
where $L^2(P\,|\,{\F_\tau})$ is the module appearing in the Hansen and Richard setup of Subsection~\ref{SS:3.5}.
The set $\GH_\tau$ is trivially closed in  $L^2(P\,|\,{\F_\tau})$. Furthemore, one has 
\begin{equation}\label{eq:231112}
\one_A\notin \GH_\tau\qquad\text{ for all $A\in \F_\tau$,\quad $P(A)>0$}
\end{equation}
since none of the payoffs $\one_A$ have zero cost in view of 
$$\E[\one_A\,\tauZ_T\,|\,\F_\tau]=\one_A\E[\tauZ_T\,|\,\F_\tau]=\one_A\neq 0.$$ 
Hence, $\tauZ_T$ yields the law of one price for wealth transfers between $\tau$ and $T$ in the complete market characterised by  $W\in L^2(P\,|\,{\F_\tau})$ being available  at the price $\E[W\,\tauZ_T\,|\,\F_\tau]$ at time $\tau$. Observe that the complete market opportunity process $\hat L$ reads
\[\hat L_\tau = \min_{W\in\GH_\tau}\E[(1-W)^2\,|\,\F_\tau] = \frac{1}{\E[\tauZ_T^2\,|\,\F_\tau]}.\]  

So far, we have only exploited properties \ref{spd.1} and \ref{spd.3} of Definition~\ref{D:SPDlop}. The time consistency condition \ref{spd.2} additionally ensures that $\GH_\sigma \supset \GH_\tau$ for $\sigma,\tau\in\T$ with $\sigma\leq\tau$. Let now $\tau$ be a predictable stopping time in $\T$ and $\tau_n$ an announcing sequence. In principle, it may happen that 
\begin{equation}\label{eq:231112b}
\one_A\in \cap_{n}\GH_{\tau_n},\qquad \text{for some } A\in \cap_{n}\F_{\tau_n},\ P(A)>0,
\end{equation}
which is somewhat surprising given that \eqref{eq:231112} holds for all $\tau\in\T$. Condition \ref{spd.4} of Definition~\ref{D:SPDlop} is needed to prevent \eqref{eq:231112b}.

We may now interpret Theorem~\ref{T:M1} as a market extension theorem: $S$ satisfies the LOP if and only if trading on $S$ can be embedded in a complete market that satisfies LOP. 

\subsection{\texorpdfstring{$L^2$}{L2} projections under LOP}\label{SS:3.4}
In Duffie and Richardson \cite{duffie.richardson.91}, the \emph{mean--variance hedging problem} seeks the mean--variance frontier of wealth distributions of the form $H+\vt\sint S_T$ over admissible $\vt$ with initial wealth 1 and a fixed contingent claim $H\in L^2$.  As in \cite{duffie.richardson.91,schweizer.92.aap}, it is convenient to approach the mean--variance hedging by first minimising the $L^2$ distance between an arbitrary contingent claim and the terminal wealth of a self-financing trading strategy,
\begin{equation}\label{MVH}
\text{ $\min_{\vt\in\TBtau}\E\left[(v+\vt\sint S_T-H)^2\,|\,\F_\tau\right]$ for $v\in L^2(\F_\tau,P)$, $\tau\in\T$.}
\end{equation}
In solving \eqref{MVH}, Theorem~\ref{T:M2} recovers and extends the main results of \v Cern\'y and Kallsen \cite{cerny.kallsen.07} ``on the general structure of mean--variance hedging'' in two ways. First, the $L^2$ projection is obtained without assuming the existence of an equivalent local martingale measure for $S$ but under the milder LOP assumption. The second novelty of Theorem~\ref{T:M2} is an expression for the conditional hedging error contained in \eqref{eq:eps2H} and \eqref{eq:mvherrorbis}, and the orthogonality statement \eqref{eq:meanzero}, which allows us (in Subsection~\ref{SS:3.6}) to formulate a \emph{conditional version} of the efficient frontier in the spirit of Hansen and Richard \cite{hansen.richard.87}. 

To prepare for the statement of Theorem~\ref{T:M2}, let us recall some notation.
Let $X$ be a special $\R^d$-valued semimartingale with predictable characteristics $(B^X,C^X,\nu^X)$; see \cite[II.2.6]{js.03}.
Denote by $\mathscr{B}^d$ the Borel $\sigma$-algebra on $\R^d$. By \cite[II.2.9]{js.03}, there exists some increasing predictable process  of integrable variation, some predictable $\R^{d\times d}$-valued process $c^X$ whose values are non-negative symmetric matrices, and some transition kernel $F^X$ from $(\Omega\times\R_+,P)$ into $(\R^d,\mathscr{B}^d)$ such that
\begin{equation*}
B^X_t=b^X\sint A_t, \ \
C^X_t=c^X\sint A_t, \ \ 
\nu^X([0,t]\times G)=F^X(G)\sint A_t\ \ \mbox{ for }t\in[0,T],\, G\in\mathscr{B}^d.\label{e:bcfa}
\end{equation*}
We call $(b^X,c^X,F^X,A)$ {\em differential characteristics} of $X$.

Especially for $A_t=t$, one can interpret $b^X_t$ as a drift rate, $c^X_t$ as a diffusion coefficient, and $F^X_t$ as a jump arrival intensity. The differential characteristics are typically derived from other ``local'' representations of the process, e.g., in terms of a stochastic differential equation. From now on, we choose the same fixed process $A$ for all the (finitely many) semimartingales in this paper. The results do not depend on the  particular choice of $A$.

If $[X,X]$ is special (i.e., $X$ is locally square integrable), 
\begin{equation*}
\tilde c^X=c^X+\int xx^\top F^X(\d x)=b^{[X,X]}.
\end{equation*}
stands for the modified second characteristic of $X$. 
If they refer to some probability measure $P^\star$ rather than $P$, we write instead $(b^{X\star},c^{X\star},F^{X\star},A)$ and
$\tilde c^{X\star}$, respectively.
We denote the joint characteristics of two special vector-valued semimartingales $X,Y$ by
$$ \left(b^{X,Y},c^{X,Y},F^{X,Y},A\right)= 
\left(\genfrac(){0pt}{0}{b^X}{b^Y},
\left(\begin{array}{cc}
c^{X} & c^{XY} \\ c^{YX} & c^{Y} 
\end{array}\right), F^{X,Y},A\right)$$
Furthermore, for locally square-integrable $X,Y$ we let
$$ \tilde c^{XY}=c^{XY}+\int xy^\top F^{X,Y}(\d x,\d y).$$
\begin{theorem}\label{T:M2}
Suppose that $S$ is locally square integrable and satisfies the law of one price (or, equivalently,  any of the conditions \ref{MT1.2}--\ref{MT1.5} in Theorem \ref{T:M1}). Then,
\begin{enumerate}[(1)]
\item\label{T:main.1} The opportunity process $L$ is the unique bounded semimartingale $L=(L_t)_{0\leq t\leq T}$ such that
\begin{enumerate}[(a)]
\item\label{T:main.1a} $L>0$, $L_->0$, and $L_T=1$.
\item\label{T:main.1b} $\frac{L}{\Exp(B^{\Log(L)})}>0$ is a martingale on $[0,T]$.
\item\label{T:main.1c} $S$ and $[S,S]$ are $P^\star$-special for the opportunity neutral measure $P^\star\sim P$ defined by  
$$\frac{dP^\star}{dP}=\frac{L_T}{\E[L_0]\Exp(B^{\Log(L)})_T}>0,$$
which implies
\begin{align*}
b^{S\star}         &= \frac{b^{S}+c^{S\mkern-2mu\Log(L)}+\int xyF^{S,\Log(L)}(\d x,\d y)}{1+\Delta B^{\Log(L)}}; \\
\tilde{c}^{S\star} &= \frac{c^{S}+\int xx^{\top }(1+y)F^{S,\Log(L)}  (\d x,\d y)}{1+\Delta B^{\Log(L)}}. 
\end{align*}
\item\label{T:main.1d} The set 
\begin{equation*}
\Xi _{a}=\argmin_{\vt\in\R^d}
\{\vt \tilde{c}^{S\star}\vt^\top-2\vt b^{S\star}\}
\end{equation*}
is non-empty.
\item\label{T:main.1e} For some or, equivalently, any $\Xi_a$-valued predictable process $a$, one has 
\begin{align}
&\frac{b^{\Log(L)}}{1+\Delta B^{\Log(L)}}=-\min_{\vt\in\R^{d}}
\{\vt \tilde{c}^{S\star}\vt^\top-2\vt b^{S\star}\}\ 
=-a\tilde{c}^{S\star}a^{\top }+2a b^{S\star},\notag\\
&-a\one_{\rrbracket\tau ,T\rrbracket}\Exp(-(a\one_{\rrbracket\tau ,T\rrbracket})\sint S)_-\in \TB_{\tau},\label{aadm}
\end{align}
for all $[0,T]$-valued stopping times $\tau $.
\end{enumerate}
\item\label{T:main.2} The optimal strategy $\vp^{(\tau)}=\vp^{(\tau)}(v,H)\in\TB_\tau$ for the conditional quadratic hedging problem \eqref{MVH} exists and is given in the feedback form by
\begin{equation}\label{eq:mvhst}
\vp^{(\tau)} (v,H)=\one_{\rrbracket \tau,T\rrbracket}\left(\xi(H) +a(V_{-}(H)-v-\vp^{(\tau)} (v,H)\sint S_{-})\right),
\end{equation}
where the \emph{mean value process} $V=(V_{t})_{0\leq t\leq T} =\big(V_{t}(H)\big)_{0\leq t\leq T}$ is given by
\begin{equation}\label{eq:V}
V_{t} =V_{t}(H)=\frac{1}{L_{t}}\E[\Exp((-\one_{\rrbracket t,T\rrbracket}a)\sint S)_{T}H\,|\,\F_{t}],
\end{equation}
$a$ is an arbitrary $\Xi_a$-valued predictable process, and the pure hedging coefficient $\xi(H)$ is an arbitrary predictable process taking values in 
\begin{equation*}
\Xi_{\xi }=\argmin_{\vt\in\R^d}\{\vt\tilde{c}^{S\star}\vt^{\top }-2\vt\tilde{c}^{SV\star}\}
\end{equation*}
with
\begin{equation*}
\tilde{c}^{SV\star} = \frac{c^{SV}+\int xz(1+y)F^{S,\Log(L),V}(\d x,\d y,\d z)}{1+\Delta B^{\Log(L)}}.
\end{equation*}
\item\label{T:main.3} 
For an arbitrary $\Xi_\xi$-valued predictable process $\xi$, let
\begin{equation}\label{eq:eps2H}
\ve^2_t(H):=\E\left[\one_{\rrbracket t,T\rrbracket}L\left(\tilde{c}^{V\star}-2\xi^\top\tilde{c}^{SV\star}
+\xi^\top \tilde{c}^{S\star}\xi\right)\sint A_T\,\big|\,\F_t\right], \quad t\in[0,T],
\end{equation}
with
\begin{equation*}
\tilde{c}^{V\star}= \frac{c^{V}+\int z^2(1+y)F^{\Log(L),V}(\d y,\d z)}{1+\Delta B^{\Log(L)}}.
\end{equation*}
Then, $\ve(H)$ is a semimartingale and for $\tau\in\T$ (resp., for predictable $\tau\in\T$), the hedging error of the optimal strategy satisfies
\begin{equation}\label{eq:mvherrorbis}
\E\big[\big(v+\vp^{(\tau)}(v,H)\sint S_{T}-H\big)^{2}\,\big|\,\F_\tau\big]=L_\tau(v-V_\tau(H))^2+\ve_\tau^2(H).
\end{equation}
We further have that $\vp^{(\tau)}(H):=\vp^{(\tau)}(V_\tau(H),H)\in\TB_\tau$ is the unique strategy in terms of uniqueness of the wealth process $\vp^{(\tau)}(H)\sint S$ (up to $P$-indistinguishability) such that
\begin{equation}\label{eq:meanzero}
\E\big[\big(V_\tau(H)+\vp^{(\tau)}(H)\sint S_{T}-H\big)(1+\vt\sint S_T)\,\big|\,\F_\tau\big]=0,\qquad \vt\in\TBtau.
\end{equation}
\end{enumerate}
Conversely, if there exists a bounded semimartingale $L=(L_t)_{0\leq t\leq T}$ satisfying \ref{T:main.1}\ref{T:main.1a}--\ref{T:main.1e}, then $L$ is the opportunity process, \ref{T:main.2} and \ref{T:main.3} hold, and $S$ satisfies the law of one price.
\end{theorem}
\begin{proof} Fix $\tau\in\T$. Recall from the proof of Theorem \ref{T:M1}\ref{MT1.6} that $\TB(\tilde{S})=\TBtau(S)$ for $\tilde{S}=\mathbbm{1}_{\rrbracket \tau,T\rrbracket }\sint S$, that $\vt\sint \tilde{S}=\vt\sint S$ for all $\vt\in\TB(\tilde{S})=\TBtau(S)$, and that the conditions \ref{MT1.1}--\ref{MT1.5} of Theorem \ref{T:M1} for $S$ imply that the conditions also hold for $\tilde{S}$. Therefore, the solution $\tilde{\vp}\in\TB(\tilde{S})$ to the \emph{unconditional} $L^2$ approximation
\begin{equation}\label{MVH:2}
\text{ $\min_{\tilde{\vt}\in\TB(\tilde{S})}\E\left[(v+\tilde{\vt}\sint \tilde{S}_T-H)^2\right], \qquad v\in L^2(\F_\tau,P)$,}
\end{equation}
 exists since $\{\tilde{\vt}\sint \tilde{S}_T\,:\,\tilde{\vt}\in\TB(\tilde{S})\}$ is closed in $L^2$ by Theorem \ref{T:M1} applied to $\tilde{S}$. By the strict convexity of the square, the terminal wealth $\tilde{\vp}\sint \tilde{S}_T$ of the optimal strategy is unique. The law of one price then yields uniqueness of the wealth process $\tilde{\vp}\sint \tilde{S}=(\tilde{\vp}\sint \tilde{S}_t)_{0\leq t\leq T}$ by Remark \ref{rem:uniqueness}.

Next, we will show that $\tilde{\vp}\in\TB(\tilde{S})=\TBtau(S)$ also optimizes \eqref{MVH}. To this end, note that the solution of \eqref{MVH:2} is uniquely characterised by the first order condition
\begin{equation*}
\E\left[(v+\tilde{\vp}\sint \tilde{S}_T-H)\tilde{\vt}\sint \tilde{S}_T\right]=0\text{ for all $\tilde{\vt}\in\TB(\tilde{S})$},
\end{equation*}
which yields 
\begin{equation*}
\E\left[(v+\tilde{\vp}\sint \tilde{S}_T-H)\tilde{\vt}\sint \tilde{S}_T~|~\F_\tau\right]=0\text{ for all $\tilde{\vt}\in\TB(\tilde{S})$}
\end{equation*}
by the definition of the conditional expectation, since
$$\text{$\tilde{\vt}\in\TBtau(\tilde{S})\iff\mathbbm{1}_F\tilde{\vt}\in\TBtau(\tilde{S})$ for all $F\in\F_\tau$.}$$ 
In view of $\TB(\tilde{S})=\TBtau(S)$ and $\vt\sint S_T= \vt\sint \tilde{S}_T$ for all $\vt\in\TBtau(S)=\TB(\tilde{S})$, the strategy $\tilde \vp$ indeed optimizes \eqref{MVH}, i.e., 
\begin{equation*}
\E\left[(v+\tilde\vp\sint S_T-H)\vt\sint S_T~|~\F_\tau\right]=0\text{ for all $\vt\in\TBtau(S)$}.
\end{equation*}

\noindent\ref{T:main.1} and the converse: As shown above, the optimizer $\vp^{(\tau)}(1,0)\in\TBtau$ exists. 
Parts (1) and (2) of \cite[Proposition~6.1]{czichowsky.schweizer.13} then yield that the opportunity process $L=(L_t)_{0\leq t\leq T}$ is the unique semimartingale satisfying \ref{T:main.1}\ref{T:main.1a}--\ref{T:main.1e}. The converse implication that a bounded semimartingale $L=(L_t)_{0\leq t\leq T}$ satisfying \ref{T:main.1}\ref{T:main.1a}--\ref{T:main.1e} is the opportunity process follows from \cite[Proposition 6.1(3)]{czichowsky.schweizer.13}.\medskip 

\noindent \ref{T:main.2} and \ref{T:main.3}: This follows by observing that one only needs to solve the unconditional $L^2$ approximation problem \eqref{MVH:2} and that the arguments of \cite[Section 4]{cerny.kallsen.07} only need the properties \ref{T:main.1}\ref{T:main.1a}--\ref{T:main.1e}. As shown above, the latter follow from the law of one price and do not require the existence of an equivalent local martingale measure with square-integrable density as assumed in \cite{cerny.kallsen.07}. 
\end{proof}
\begin{remark}\label{R:tau-}
All statements about trading from $\tau$ onwards in Theorem~\ref{T:M2} have a natural counterpart for trading from $\tau-$ for predictable $\tau\in\T$, using strategies in $\TB_{\tau-}$. For example, for  $v\in L^2(\F_{\tau-},P)$ one obtains 
\begin{align*}
\E\big[\big(v+\vp^{(\tau-)}(v,H)\sint S_{T}-H\big)^{2}\big|\,\F_{\tau-}\big]={}&L_{\tau-}(v-V_{\tau-}(H))^2+\ve_{\tau-}^2(H),\\
\E\big[\big(V_{\tau-}(H)+\vp^{(\tau-)}(H)\sint S_{T}-H\big)(1+\vt\sint S_T)\big|\,\F_{\tau-}\big]={}&0,\qquad \vt\in\TB_{\tau-},
\end{align*}
thus increasing the scope of Theorem~\ref{T:M2}. The proofs for $\tau-$ are completely analogous and therefore omitted. 
\end{remark}
\begin{remark}
When the law of one price for $S$ fails, the $L^2$ approximation problem \eqref{MVH1} becomes degenerate in the following sense. Let 
$$\sigma:=\inf\{t>0\,:\,L_{t-}=0\};\qquad\tau:=\inf\{ t>0\,:\,L_t=0\}.$$
Then, at least one of the events $\{\sigma\leq T\}$ and $\{\tau<T\}$ has positive probability and \emph{any} $L^2$ random variable supported on these events can be approximated with arbitrary precision in $L^2$ by zero-cost strategies in $\TB$. In such case, $\{\vt\sint S_T:\vt\in\TB\}$ may or may not be closed as illustrated by Examples 3.10, 3.11, and 3.12 and Theorem 5.3 of Delbaen et al. \cite{delbaen.al.97}. 
\end{remark}
Combining Theorems \ref{T:M1} and \ref{T:M2} allows us to give a simplified verification theorem for the law of one price that only involves finding a candidate $\hat{L}=(\hat{L}_t)_{0\leq t\leq T}$ for the opportunity process and not the actual opportunity process $L=(L_t)_{0\leq t\leq T}$. Its significance compared to earlier results based on the converse implication of Theorem~\ref{T:M2} is that we do not need to verify the admissibility of the adjustment process $a$ in \eqref{aadm}, which is usually a difficult task as illustrated in the counterexample \cite{cerny.kallsen.08a}.
\begin{theorem}\label{vt}
Let $S$ be locally square integrable. Suppose that there exists a bounded semimartingale $\hat{L}=(\hat{L}_t)_{0\leq t\leq T}$ such that
\begin{enumerate}[(a)]
\item\label{T:veri.1a} $\hat{L}>0$, $\hat{L}_->0$, and $\hat{L}_T=1$.
\item\label{T:veri.1b}  For
\begin{align*}
\hat{b}         &= \frac{b^{S}+c^{S\mkern-2mu\Log(\hat{L})}+\int xyF^{S,\Log(\hat{L})}(\d x,\d y)}{1+\Delta B^{\Log(\hat{L})}}; \\
\hat{c} &= \frac{c^{S}+\int xx^{\top }(1+y)F^{S,\Log(\hat{L})}  (\d x,\d y)}{1+\Delta B^{\Log(\hat{L})}}, 
\end{align*}
one has 
\begin{align*}
&\frac{b^{\Log(\hat{L})}}{1+\Delta B^{\Log(\hat{L})}}=-\min_{\vt\in\R^{d}}
\{\vt \hat{c}\vt^\top-2\vt \hat{b}\}.
\end{align*}
\end{enumerate}
Then $S$ satisfies the law of one price (or, equivalently,  any of the conditions \ref{MT1.2}--\ref{MT1.5} in Theorem \ref{T:M1}). In particular, the opportunity process $L=(L_t)_{0\leq t\leq T}$ is the maximal bounded semimartingale satisfying \ref{T:veri.1a} and \ref{T:veri.1b}.
\end{theorem}
\begin{proof}
By properties \ref{T:veri.1a} and \ref{T:veri.1b}, the semimartingale $\hat{L}$ is a solution to the BSDE (4.18) in \cite{czichowsky.schweizer.13} that takes the form (6.8) without cone constraints. Because the opportunity process $L$ is the maximal solution to this BSDE by \cite[Lemma 4.17]{czichowsky.schweizer.13}, it follows that the opportunity process $L$ satisfies $L\geq \hat{L}$. Therefore, $L>0$ and $L_->0$ by \ref{T:veri.1a} so that \ref{MT1.5} and hence any of the conditions \ref{MT1.1}--\ref{MT1.4} in Theorem \ref{T:M1} holds. 
\end{proof}


\subsection{Hansen and Richard (1987) framework}\label{SS:3.5}
The results in Subsections~\ref{SS:3.1} and \ref{SS:3.4} have natural counterparts for trading in a wider class of admissible strategies $\TTtau$, where one only requires \emph{conditional} square integrability of the terminal wealth. This will allow us to study the Hansen and Richard \cite{hansen.richard.87} framework enhanced by dynamic trading in Subsection~\ref{SS:3.6}. 

\begin{definition}\label{D:TT}
For $\tau\in\T$, a trading strategy $\vt\in L(S)$ is in $\TTtau$, if $\vt=0$ on $\llbracket 0,\tau\rrbracket$ and there exists a sequence $(\vt^n)_{n\in\N}$ of simple trading strategies $\vt^n=(\vt^n_t)_{0\leq t\leq T}\in\Theta_\tau$ such that
\begin{enumerate}[(1)]
\item\label{TTtau.1} $\vt^n\sint S_T\xrightarrow{ L^2(P\,|\,{\F_\tau})} \vt\sint S_T$, that is, $\E[(\vt^n\sint S_T-\vt\sint S_T)^2\,|\,\F_\tau]\xrightarrow{L^0}0$.
\item $\vt^n\sint S_\sigma\xrightarrow{L^0} \vt\sint S_\sigma$ for all $\sigma\in\T$.
\end{enumerate}
For predictable $\tau\in\T$, the set $\TT_{\tau-}$ is defined analogously by replacing $\tau$ with $\tau-$ above and requiring $\vt=0$ only on $\llbracket 0,\tau\llbracket$. 
\end{definition}
\begin{remark}\label{LPG}
Convergence \ref{TTtau.1} has an equivalent form
\begin{enumerate}
\myitem{(1\/\hspace{0.05em}$'$)}\label{TTtau.1'} There is a bounded $0<K\in L^0(\F_\tau,P)$ such that $K\vt^n\sint S_T\xrightarrow{ L^2} K\vt\sint S_T$. 
\end{enumerate} 
It follows immediately from \ref{TTtau.1'} that $\TTtau$ satisfies
\begin{equation}\label{TTtau}
\TTtau = L^0(\F_\tau,P)\TBtau,\qquad \tau\in\T.
\end{equation}
To see that $K$ in \ref{TTtau.1'} exists, let $X_n=\vt^n\sint S_T$ with $\vt^n\in\TTtau$ such that $X_n\xrightarrow{ L^2(P\,|\,{\F_\tau})} X=\vt\sint S_T$. By passing to a subsequence, we may assume that $\E[(X_n-X)^2\,|\,\F_\tau]\xrightarrow{\text{$P$-a.s.}}0$. Since $\E[X_n^2\,|\,\F_\tau]<\infty$, this yields $\E[X^2\,|\,\F_\tau]<\infty$ and $\sup_{n\in\N}\E[X_n^2\,|\,\F_\tau]<\infty$. Hence, \ref{TTtau.1'} holds with  
$$K=\frac{1}{\sqrt{1+\sup_{n\in\N}\E[X_n^2\,|\,\F_\tau]+\E[X^2\,|\,\F_\tau]}}\in(0,1).$$ 
\noindent``\ref{TTtau.1'} $\Rightarrow$ \ref{TTtau.1}'' follows from the properties of conditional expectation \cite[Theorem~I.1.21]{he.wang.yan.92}. 
\end{remark}
By Theorem \ref{T:M1}\ref{MT1.6}, the law of one price for the price process $S$ yields $L^2$-closedness of admissible terminal wealths 
$\{\vt\sint S_T\,:\,\vt\in\TBtau\}$ for all $\tau\in\T$. The next result translates the $L^2$-closedness of $\{\vt\sint S_T\,:\,\vt\in\TBtau\}$ into the conditional closedness of terminal wealths generated by the wider class $\TTtau$. To avoid repetition, we do not state an analogous result for $\TT_{\tau-}$ for predictable $\tau\in\T$.
\begin{theorem}\label{T:GTtau}
Assume the law of one price for $S$. Then the following statements hold for any $\tau\in\T$, $H\in L^2(P\,|\,\F_\tau)$, and $v\in L^0(\F_\tau,P)$.
\begin{enumerate}[(1)]
\item\label{GTtau.1} The set of terminal wealths $\{\vt\sint S_T\,:\,\vt\in\TTtau\}$ is closed in $L^2(P\,|\,\F_\tau)$. 
\item\label{GTtau.2} The optimizer in $\min_{\vt\in\TTtau}\E\left[(v+\vt\sint S_T-H)^2\,|\,\F_\tau\right]$
is the unique (up to null a strategy) process $\vp^{(\tau)}(v,H)$ satisfying
\begin{equation*}
\E\big[\big(v+\vp^{(\tau)}(v,H)\sint S_{T}-H\big)\vt\sint S_T\,\big|\,\F_\tau\big]=0,\qquad \vt\in\TTtau.
\end{equation*} 
\item\label{GTtau.3} For $H\in L^2(\F_T,P)$ and $L_\tau v^2\in L^1(\F_\tau,P)$, the optimal hedges in $\TBtau$ and $\TTtau$ coincide. In particular, the opportunity process over $\TB$ coincides with that over $\TT$.  
\item\label{GTtau.4} The tracking process $V(H)$ is well-defined by formula \eqref{eq:V} on $\llbracket \tau,T\rrbracket$. The feedback formula for the optimal strategy \eqref{eq:mvhst} remains in force on $\llbracket \tau,T\rrbracket$.
\item\label{GTtau.5} The minimal hedging error in \eqref{eq:eps2H} is well-defined on $\llbracket \tau,T\rrbracket$. Furthermore, formula \eqref{eq:mvherrorbis} continues to hold.
\item\label{GTtau.6} \emph{Mutatis mutandis}, Remark~\ref{R:tau-} applies over the wider set $\TT$. 
\end{enumerate}
\end{theorem}
\begin{proof}
\ref{GTtau.1}: Let 
\begin{equation*}
\cGtau:= \{\vt\sint S_T\,:\,\vt\in\Theta_\tau\};\qquad  \GB_\tau:= \{\vt\sint S_T\,:\,\vt\in\TBtau\}
;\qquad  \GT_\tau:= \{\vt\sint S_T\,:\,\vt\in\TTtau\}.
\end{equation*}
Consider a sequence $(X_n)_{n\in\N}$ converging in $L^2(P\,|\,\F_\tau)$ to some $X$.  By \eqref{TTtau} we have $\GTtau=L^0(\F_\tau,P)\GBtau$. By Remark~\ref{LPG}, there is a bounded positive $K\in L^0(\F_\tau,P)$ such that $KX_n$ in $(L^0(\F_\tau,P)\cGtau)\cap L^2=\cGtau$ converges to $KX$ in $L^2$. Hence, $KX\in\cl\cGtau = \GBtau$ and $X\in \frac{1}{K}\GBtau\subset\GTtau$.\medskip

\noindent{\ref{GTtau.2}--\ref{GTtau.6}}: There is a bounded positive $K\in L^0(\F_\tau,P)$ such that $Kv$ and $KH$ are in $L^2$. It is then immediate from \eqref{TTtau} that $\vp^{(\tau)}(Kv,KH) = K \vp^{(\tau)}(v,H)$ and the formulae follow by applying Theorem~\ref{T:M2} to the pair $(Kv,KH)$.  
\end{proof}


\subsection{Mean--variance hedging}\label{SS:3.6}
The next theorem extends the mean--variance hedging of Duffie and Richardson \cite[Subsection~4.3]{duffie.richardson.91} to general contingent claims and conditional mean--variance frontier. Observe that even in the classical case with trivial $\F_0$  and square-integrable $H$, our characterization of the mean--variance frontier simplifies that of \cite{duffie.richardson.91}. By not requiring $L_\tau<1$, we also expand the conditional mean--variance analysis of Hansen and Richard \cite{hansen.richard.87} to a setting without their Assumption~3.1. Observe that dynamic trading has not  previously been considered in the setting of \cite{hansen.richard.87}. All statements below have natural counterparts for trading from $\tau-$ for predictable $\tau\in\T$ in the spirit of Remark~\ref{R:tau-}. 

\begin{theorem}\label{T:EF1}
Assume~$S$ satisfies the law of one price. For $\tau\in\T$ and $H\in L^2(P\,|\,\F_\tau)$, the following are equivalent.
\begin{enumerate}[(i)]
\item\label{EF.1} $W\in H+\{\vt\sint S_T\,:\,\vt\in\TTtau\}$ has the smallest conditional variance for a given conditional mean.
\item\label{EF.2} $W = H-\vp^{(\tau)}(H)\sint S_T+(\lambda-V_\tau(H))(1-\Exp((-\one_{\rrbracket t,T\rrbracket}a)\sint S)_{T})$ for some $\lambda\in L^0(\F_\tau,P)$.
\item\label{EF.3} One has
\begin{equation*}
\begin{split}
\Var(W\,|\,\F_\tau) :={}& \E[W^2\,|\,\F_\tau]- (\E[W\,|\,\F_\tau])^2 \\
={}& \ve_\tau^2(H)+\frac{L_\tau}{1-L_\tau}\left( \E[W\,|\,\F_\tau]-V_\tau(H)\right) ^{2},
\end{split}
\end{equation*}
with the convention $0/0=0$. In particular, on the set $\{L_\tau=1\}$, the efficient frontier collapses to a single point with $\E[W\,|\,\F_\tau] = V_\tau(H)$ and $\Var(W\,|\,\F_\tau) = \ve_\tau^2(H)$.
\end{enumerate}
\end{theorem}
\begin{proof}
Since $\vp^{(\tau)}(0,1)$ maximizes conditional expected quadratic utility in $\TBtau$, it is mean--variance efficient in $\TBtau$. Letting $X_\tau:=\vp^{(\tau)}(0,1)\sint S_T$, one has $\E[X_\tau\,|\,\F_\tau]=\E[X^2_\tau\,|\,\F_\tau]=1-L_\tau$ in view of the orthogonality \eqref{eq:meanzero} and the identity
\begin{equation}\label{eq:Xtau}
X_\tau = 1-\Exp((-\one_{\rrbracket t,T\rrbracket}a)\sint S)_{T}.
\end{equation}
By orthogonality \eqref{eq:meanzero}, $\GTtau = \{\vt\sint S_T\,:\,\vt\in\TTtau\}$ decomposes orthogonally into $L^0(\F_\tau,P)X_\tau$ and a subspace whose elements necessarily have conditional mean zero. This shows that all efficient payoffs in $\GTtau$ are of the form $\lambda X_\tau$ for
$\lambda\in L^0(\F_\tau,P)$. Writing 
$$H+\GTtau = \underbrace{H-V_\tau(H)-\vp^{(\tau)}(H)\sint S_T}_A + \underbrace{V_\tau(H)(1-X_\tau)}_B + \underbrace{\GTtau}_C,$$
the equivalence of \ref{EF.1} and \ref{EF.2} now follows from \eqref{eq:meanzero} and \eqref{eq:Xtau} in view of the mutual orthogonality of $A$, $B$, and $C$. The orthogonality yields moment expressions for $W$ in \ref{EF.2},
\begin{align*}
\E[W\,|\,\F_\tau] ={}& \lambda(1-L_\tau)+V_\tau(H)L_\tau,\\
\E[W^2\,|\,\F_\tau] ={}& \ve_\tau^2(H)+\lambda^2(1-L_\tau)+V^2_\tau(H)L_\tau,
\end{align*}
which after algebraic manipulations shows the equivalence of \ref{EF.2} and \ref{EF.3}.
\end{proof}
We next examine the mean--variance frontier when the contingent claim is not part of the endowment but instead can be purchased at time $\tau$ for the price $\pi\in L^0(\F_\tau,P)$. The question is then how to select an amount of the contingent claim to be held from $\tau$ to maturity $T$ so as to maximize the Sharpe ratio of a zero-cost position that involves dynamic trading in the underlying assets $(1,S)$ and a static trade in the contingent claim.
\begin{definition}
For $\tau\in\T$, $\pi\in L^0(\F_\tau,P)$, and $H\in L^2(P\,|\,\F_\tau)$ we call
\begin{equation*}
\rho_\tau:=\sup\left\{\frac{\E[(\vt\sint S_T+\lambda(\pi-H)\,|\,\F_\tau]}{\sqrt{\Var(\vt\sint S_T+\lambda(\pi-H)\,|\,\F_\tau)}}:
\vt\in\Theta_\tau,\lambda\in L^0(\F_\tau,P)\right\}
\end{equation*}
the \emph{maximal Sharpe ratio on $\rrbracket \tau,T\rrbracket$}, with the convention $0/0=0$.  
\end{definition}
\begin{theorem}\label{T:sharpe} 
Assume $S$ satisfies the law of one price. Fixing $\tau\in\T$ and $H\in L^2(P\,|\,\F_\tau)$, suppose further that at time $\tau$, the contingent claim $H$ delivered at $T$ is available at price $\pi\in L^0(\F_\tau,P)$, to be held to maturity. Assume this $\pi$ satisfies $\one_{\{\ve_\tau(H)=0\}}(\pi-V_\tau(H))=0$. Then the maximal conditional Sharpe ratio $\rho_\tau$ is given by 
\begin{equation}\label{eq:SRoption}
\rho_\tau^{2}= L_\tau^{-1}-1+\frac{\left(\pi-V_{\tau}(H)\right) ^{2}}{\ve _{\tau}^{2}(H)},
\end{equation}
with the convention $0/0=0$. Furthermore, this Sharpe ratio is attained by the terminal wealth
\begin{equation}\label{eq:SRwealth}
\hat\eta(\pi-H)+\vp^{(\tau)}(0,1-\hat\eta (\pi-H))\sint S_T,
\end{equation}
with the contingent claim position
\begin{equation}\label{eq:hateta}
\hat\eta =\frac{\pi-V_{\tau}(H)}{\varepsilon _{\tau}^{2}(H)}\frac{1}{1+\rho_\tau^{2}}.
\end{equation}
\end{theorem}
\begin{proof}
The proof mirrors \cite[Lemma~5.1]{cerny.kallsen.08b}. Let $X^{\eta,\vt}:=\eta \left( \pi-H\right)+\vt \sint S_{T}$ 
for $\vt\in\TBtau$ and $\F_\tau$-measurable $\eta$ and $\pi$. For conditionally square-integrable $X$, one obtains easily
\begin{align*}
\frac{(\E[X\,|\,\F_\tau])^{2}}{\Var(X\,|\,\F_\tau)}
={}\frac{1}{\inf_{\alpha \in L^0(\F_\tau,P)}\{\E\left[ (1-\alpha X)^{2}\,|\,\F_\tau\right] \}}-1.
\end{align*}
Then 
\begin{align*}
\rho_\tau^{2} ={}&\sup_{\eta \in L^0(\F_\tau,P),\vartheta \in \TBtau}
\left\{\frac{(\E[X^{\eta,\vt}\,|\,\F_\tau])^{2}}{\Var(X^{\eta,\vt}\,|\,\F_\tau)}\right\}=\sup_{\alpha,\eta \in L^0(\F_\tau,P),\vartheta \in \TBtau}
\left\{ \frac{1}{\E\left[ (1-\alpha X^{\eta,\vt})^{2}\,|\,\F_\tau\right] }-1\right\} \\
={}&\frac{1}{\inf_{\eta \in L^0(\F_\tau,P)}\left\{ \inf_{\vartheta \in \TBtau}\{\E\left[ (1-X^{\eta,\vt})^{2}\,|\,\F_\tau\right] \}\right\} }-1 \\
={}&\frac{1}{\inf_{\eta \in L^0(\F_\tau,P)}\left\{ L_{\tau}(1-\eta(\pi-V_{\tau}(H)))^{2}+\eta ^{2}\varepsilon _{\tau}^{2}(H)\right\} }-1,
\end{align*}
where the last equality follows from \eqref{eq:mvherrorbis}  with the contingent claim $1-\eta (\pi-H)$. Straightforward calculations yield the optimal volume sold in \eqref{eq:hateta} and the maximal conditional Sharpe ratio \eqref{eq:SRoption}. By \eqref{eq:mvhst} with the contingent claim $1-\hat\eta(\pi-H)$, the optimal investment-cum-hedging wealth is given by \eqref{eq:SRwealth}.
\end{proof}
\begin{remark}[Absence of arbitrage in an extended market]
A square-integrable $\Exp$-density compatible with $S$ represents an extended market that embeds admissible trading in $S$, i.e.,
\begin{equation*}
\vt\sint S_t=\E[(\vt\sint S_T)\,{}^t\!\Exp(N)_T\,|\,\F_t],\qquad 0\leq t\leq T,\ \vt\in\TB.
\end{equation*}
The extended market trades every contingent claim $H\in L^2$ at the price
\begin{equation*}
\tilde S_t:=\E[H\, {}^{t}\!\Exp(N)_T\,|\,\F_t],\qquad 0\leq t\leq T.
\end{equation*} 
If $\Exp(N)>0$, then $\Exp(N)$ is the density of an equivalent local martingale measure for $S$ and $\tilde S$, hence the extended market $(S,\tilde S)$ is arbitrage-free over admissible strategies. However, as soon as $\Exp(N)\leq0 $ with positive $P$-probability, the extended market will contain arbitrage opportunities since the contingent claim $H=\one_{\{\tau\leq T\}}$, where $\tau$ is the first time $\Exp(N)$ is less than or equal to zero, trades at a non-positive price $\E[\Exp(N)_\tau \one_{\{\tau\leq T\}}\,|\,\F_0]$ at time $0$. 
\end{remark}
\begin{remark}[Minimality of the variance-optimal market completion]\label{R:VOED}
Fix $\tau\in\T$ and consider a statically complete market for wealth transfers between $\tau$ and $T$, where every terminal wealth distribution $W\in L^2(P\,|\,\F_\tau)$ is available to purchase at time $\tau$ at the price
$$ p_\tau(W) = \E[\tauE(N)_TW\,|\,\F_\tau]$$
with $N=-a\sint S +\Log(L) - [a\sint S ,\Log(L)]$. 
Such market subsumes dynamic trading in $S$ using strategies in $\TTtau$ as well as static positions in any contingent claim $H\in L^2(P\,|\,\F_\tau)$ at the price $V_\tau(H)$. It is well known (see, e.g., Hansen and Jagannathan \cite[Eq.~17]{hansen.jagannathan.91}) that the highest Sharpe ratio attainable in such a complete market takes the value
$$\Var(\tauE(N)_T\,|\,\F_\tau) = \E\bigg[\frac{\tauE(-a\sint S)^2_T}{L_\tau^2}\,\bigg|\,\F_\tau\bigg] = L_\tau^{-1}-1.$$
Thus, the variance-optimal $\Exp$-density $\Exp(N)$ is the only square-integrable state price density compatible with $S$ whose complete market does not expand the conditional efficient frontiers generated by trading in $S$ alone; cf. Theorem~\ref{T:sharpe}. 
\end{remark}


\section{Proofs}\label{S:4}


\subsection{Relation between state price densities and \texorpdfstring{$\Exp$}{E}-densities}
We begin with two propositions that are of more general interest. Recall that, by a conditional version of the Riesz representation theorem for Hilbert spaces in \cite[Theorem~2.1]{hansen.richard.87}, a conditionally linear and continuous operator can be represented as a conditional expectation involving a conditionally square-integrable random variable. More precisely, for $\tau\in\T$ and an $\F_\tau$-conditionally linear and continuous operator $p_{\tau}:L^2(\F_T,P)\to L^0(\F_\tau,P)$, there is an $\F_T$-measurable random variable $\tauZ_T$ such that 
\begin{equation}\label{eq:ptau2}
\E[\tauZ_T^2\,|\,\F_\tau]<\infty\quad\text{ and }\quad p_\tau(H) = \E[\tauZ_T H \,|\,\F_\tau] \quad\text{ for all } H\in L^2(\F_T,P).
\end{equation}
The conditional operator norm of $p_{\tau}$ then satisfies $\|p_{\tau}\|:=\left(\E[\tauZ_T^2\,|\,\F_{\tau}]\right)^{1/2}$. We refer to Cerreia-Vioglio, Kupper, Maccheroni, Marinacci, and Vogelpoth~\cite{cv.al.16.jmaa} for more details about conditional $L^p$ spaces and an analysis of conditionally linear operators on them via $L^0$ modules.

The next proposition shows that the existence of a family $\{p_\tau\}_{\tau\in\T}$ of pricing operators and a family $\{\tauZ_T\}_{\tau\in\T}$ of state price densities, both satisfying the law of one price, are equivalent. Comparing the Definitions \ref{D:PSlop} and \ref{D:SPDlop}, the equivalence of properties \ref{D:PSlop:1}--\ref{D:PSlop:3b} of Definition~\ref{D:PSlop} and \ref{spd.1}--\ref{spd.3} of Definition~\ref{D:SPDlop} directly follows from using \eqref{eq:ptau2}. Therefore, we only need to show the equivalence of properties \ref{spd.4} of Definition~\ref{D:PSlop} and \ref{spd.4} of Definition~\ref{D:SPDlop}, which is a consequence of the conditional version of the uniform boundedness principle below.
\begin{proposition}\label{P:UBP}
 For a predictable stopping time $\sigma\in\T$ with an announcing sequence $(\sigma_m)_{m=1}^\infty\in\T$, the following are equivalent.
\begin{enumerate}[(i)]
\item\label{ubp1} There are $\F_{\sigma_m}$-conditionally linear and continuous mappings $p_{\sigma_m}:L^2(\F_T,P)\to L^0(\F_{\sigma_m},P)$ that are pointwise convergent, that is, $(p_{\sigma_m}(H))_{m=1}^\infty$ is a convergent sequence in $L^0$ for all $H\in L^2(\F_T,P)$.
\item\label{ubp2} There is an $\F_{\sigma-}$-linear and continuous mapping $p_{\sigma-}:L^2(\F_T,P)\to L^0(\F_{\sigma-},P)$ such that $p_{\sigma_m}(H)\xrightarrow{L^0}p_{\sigma-}(H)$, as $m\to\infty$, for all $H\in L^2(\F_T,P)$.
\item\label{ubp3} There is a sequence $\{\sigmZ_T\}_{m=1}^\infty$ of random variables such that $p_{\sigma_m}$ are given by \eqref{eq:ptau2}, the ``conditional operator norms'' $\|p_{\sigma_m}\|:=\left(\E[(\sigmZ_T)^2\,|\,\F_{\sigma_m}]\right)^{1/2}$ are bounded, that is, 
\begin{equation}\label{eq:ubp}
C:=\sup_{m\in\N}\left(\E[(\sigmZ_T)^2\,|\,\F_{\sigma_m}]\right)^{1/2}<\infty,
\end{equation}
and there is $D$, a dense subset of $L^2(\F_T,P)$, such that for all $H\in D$ the sequence  $\left(p_{\sigma_m}(H)\right)_{m=1}^\infty$ converges in $L^0$ to some random variable.
\end{enumerate}
\end{proposition}
\begin{proof} Trivially, \ref{ubp2} implies \ref{ubp1}. To see that \ref{ubp3} gives \ref{ubp2}, set $p_{\sigma-}(H):=\lim_{m\to\infty}p_{\sigma_m}(H)$ for all $H\in D$. Then, $p_{\sigma-}:D\to L^0(\F_{\sigma-},P)$ is a continuous linear mapping. Indeed, for $H_1,H_2\in D$, we have that
\begin{align*}
|p_{\sigma-}(H_1)-p_{\sigma-}(H_2)|&{}=|\lim_{m\to\infty}p_{\sigma_m}(H_1)-\lim_{m\to\infty}p_{\sigma_m}(H_2)|\\
&{}=\lim_{m\to\infty}|p_{\sigma_m}(H_1-H_2)|\\
&{}=\lim_{m\to\infty}|\E[\sigmZ_T(H_1-H_2)\,|\,\F_{\sigma_m}]|\\
&{}\leq \lim_{m\to\infty}\left(\E[(\sigmZ_T)^2\,|\,\F_{\sigma_m}]\right)^{1/2}\left(\E[(H_1-H_2)^2\,|\,\F_{\sigma_m}]\right)^{1/2}\\
&{}\leq C\left(\E[(H_1-H_2)^2\,|\,\F_{\sigma-}]\right)^{1/2},
\end{align*}
where we use that $(\F_{\sigma_m})_{m=1}^\infty$ is increasing and $\bigvee_{m=1}^\infty\F_{\sigma_m}=\F_{\sigma-}$, which also implies that the range of $p_{\sigma-}$ is $L^0(\F_{\sigma-},P)$. Since $D$ is dense in $L^2(\F_T,P)$ and $L^0(\F_{\sigma-},P)$ is complete, we can therefore extend $p_{\sigma-}$ from $D$ to $L^2(\F_T,P)$ by continuity. Note that this implies that $p_{\sigma-}(H)=\lim_{m\to\infty}p_{\sigma_m}(H)$ for all $H\in L^2(\F_T,P)$ and hence that $p_{\sigma-}$ is $\F_{\sigma_m}$-linear for all $m\in\N$. The $\F_{\sigma-}$-linearity then follows from the continuity of $p_{\sigma-}$ and the fact that the indicator function $\one_A$ of every set $A\in\F_{\sigma-}$ can be approximated by the indicator functions $\one_{A_m}$ of sets $A_m\in\F_{\sigma_m}$ in $L^0$ by Caratheodory's extension theorem, since $\bigvee_{m=1}^\infty\F_{\sigma_m}=\F_{\sigma-}$.

The proof that \ref{ubp1} yields \ref{ubp3} is a modification of a direct proof of the uniform boundedness principle in Tao \cite[Remark~1.7.6]{tao.10}. By way of contradiction, we suppose that \eqref{eq:ubp} fails. Then, there is a subsequence $(m_k)_{k=1}^\infty$ such that
$$\lim_{k\to\infty}\left(\E[(\sigmkZ_T)^2\,|\,\F_{\sigma_{m_k}}]\right)^{1/2}=\infty$$
on
$A:=\{\sup_{m\in\N}\left(\E[(\sigmZ_T)^2\,|\,\F_{\sigma_m}]\right)^{1/2}=\infty\}$ with $P(A)>0$. By passing to a further subsequence still denoted by $(m_k)_{k=1}^\infty$, we can assume that
\begin{equation}
\lim_{k\to\infty}\left(\E[(\sigmkZ_T)^2\,|\,\F_{\sigma_{m_k}}]\right)^{1/2}\one_{A_k}=\infty\label{pr:ubp:1}
\end{equation}
on $A$ for $A_k:=\left\{\left(\E[(\sigmkZ_T)^2\,|\,\F_{\sigma_{m_k}}]\right)^{1/2}\geq 100^k\right\}\in\F_{\sigma_{m_k}}$.

Set $\tilde{H}_k:=\frac{\sigmkZ_T}{\left(\E[(\sigmkZ_T)^2\,|\,\F_{\sigma_{m_k}}]\right)^{1/2}}.$ Then, $\tilde{H}_k$ is in $L^2(\F_T,P)$ with $\left(\E[(\tilde{H}_k)^2\,|\,\F_{\sigma_{m_k}}]\right)^{1/2}=1$. Therefore, we have that $H=\sum_{k=1}^\infty \ve_k 10^{-k} \tilde{H}_k\in L^2$ for all choices of signs $\ve_k\in\{-1,+1\}$. Moreover, we can choose the signs $\ve_k\in\{-1,+1\}$ successively in the following way. After $\ve_1,\ldots,\ve_{k-1}$ have been fixed, we select $\ve_k\in\{-1,+1\}$ such that
$$\E\left[\sigmkZ_T\left(\sum_{n=1}^k \ve_n 10^{-n} \tilde{H}_n\right)\,\Big|\,\F_{\sigma_{m_k}}\right]\geq10^{-k} \left(\E[(\sigmkZ_T)^2\,|\,\F_{\sigma_{m_k}}]\right)^{1/2}.$$
By the Cauchy-Schwarz and the triangular inequality, this implies
\begin{align}
\E\left[\sigmkZ_T H\,|\,\F_{\sigma_{m_k}}\right]&{}=\E\left[\sigmkZ_T\left(\sum_{n=1}^k \ve_n 10^{-n} \tilde{H}_n+\sum_{n=k+1}^\infty \ve_n 10^{-n} \tilde{H}_n\right)\,\Big|\,\F_{\sigma_{m_k}}\right]\nonumber\\
&\geq10^{-k} \left(\E[(\sigmkZ_T)^2\,|\,\F_{\sigma_{m_k}}]\right)^{1/2}-\left(\E[(\sigmkZ_T)^2\,|\,\F_{\sigma_{m_k}}]\right)^{1/2}\left(\sum_{n=k+1}^\infty 10^{-n}\right)\nonumber\\
&=10^{-k}(1-1/9)\left(\E[(\sigmkZ_T)^2\,|\,\F_{\sigma_{m_k}}]\right)^{1/2}.\label{pr:ubp:2}
\end{align}
But, since $10^{-k}\left(\E[(\sigmkZ_T)^2\,|\,\F_{\sigma_{m_k}}]\right)^{1/2}\geq 10^k$ on $A_k$, this contradicts the convergence of $p_{\sigma_{m_k}}(H)=\E\left[\sigmkZ_T H\,|\,\F_{\sigma_{m_k}}\right]$ in $L^0$ by \eqref{pr:ubp:1} and \eqref{pr:ubp:2}, which completes the proof.
\end{proof}
The next proposition shows that a state price density $\{\tauZ_T\}_{\tau \in \T}$ satisfying the law of one price can be represented in terms of the stochastic exponentials of a local martingale, more precisely, as a square-integrable $\Exp$-density.
\begin{proposition}\label{P:sopZ}
For a family of random variables $\{\tauZ_T\}_{\tau \in \T}$, the following are equivalent.
\begin{enumerate}[(i)]
	\item\label{sopZ.1} The family $\{\tauZ_T\}_{\tau \in \T}$ is a state price density satisfying the law of one price (i.e., properties \ref{spd.1}--\ref{spd.4} of Definition~\ref{D:SPDlop}).
	\item\label{sopZ.3} There is a locally square-integrable martingale $N$ such that $\Exp(N)$ is a square-integrable $\Exp$-density and, for all $\tau \in \T$, one has $\tauZ_T=\tauE(N)_T$. 
\end{enumerate}
\end{proposition}
\begin{proof}
\noindent ``\ref{sopZ.1} $\Rightarrow$ \ref{sopZ.3}'': We adapt the proof technique of \cite{delbaen.schachermayer.96.bej} as given in \cite[Lemma 3.5]{czichowsky.schweizer.13} to our situation. For this, fix one $\tau\in\T$, define $\tauZ=(\tauZ_{t})_{0\leq t\leq T}$ by $\tauZ_{t}=\E[\tauZ_T\,|\,\F_t]$ for $0\leq t\leq T$ and choose $\{\sigma_n\}_{n\in\N}$ and $\sigma$ as follows,
\[\sigma=\inf\{ t>0\,:\,\tauZ_{t}=0\};\qquad \sigma_n=\Inf\left\{t>0\,:\,|\tauZ_t|\leq \frac{1}{n+1}\right\}\wedge T,\qquad n\in\N.\]
Next, let $F=\{\tauZ_{\sigma_-}=0\}$, observing that $\tauZ$ is a locally square-integrable local martingale by property \ref{spd.3} of Definition~\ref{D:SPDlop} and hence $\tauZ_{\sigma_-}$ is well defined. Note that $\tauZ_{t}(\omega)=1$ for $0\leq t\leq\tau(\omega)$ by property \ref{spd.1} of Definition~\ref{D:SPDlop}. The Cauchy--Schwartz inequality yields
\begin{equation}\label{eq:220509}
\begin{split}
\one_{\{\sigma_n <\sigma\}}&{}=
\E\big[{}^{\sigma_n}\!Z_T\one_{\{\sigma_n <\sigma\}}\,\big|\,\F_{\sigma_n}\big]
=\E\big[{}^{\sigma_n}\!Z_T\one_{\{\sigma_n <\sigma\}}\one_{F^c}\,\big|\,\F_{\sigma_n}\big]\\
&{}\leq \E\big[{}^{\sigma_n}\!Z_T^2\one_{\{\sigma_n <\sigma\}}\,\big|\,\F_{\sigma_n}\big]^{\frac{1}{2}}
P[F^c\,|\,\F_{\sigma_n}]^{\frac{1}{2}}=\E\big[{}^{\sigma_n}\!Z_T^2\,\big|\,\F_{\sigma_n}\big]^{\frac{1}{2}}
P[F^c\,|\,\F_{\sigma_n}]^{\frac{1}{2}}. 
\end{split}
\end{equation}
On sending $n\to\infty$, one obtains in view of \eqref{eq:220509} and the bounded conditional second moments at predictable stopping times (by property \ref{spd.4} of Definition~\ref{D:SPDlop}) that
\begin{align*}
\one_F
&{}=
\Lim_{n\to\infty}\one_{\{\sigma_n <\sigma\}}\one_F\leq\limsup_{n\to\infty}\left(\E\big[{}^{\sigma_n}\!Z_T^2\,\big|\,\F_{\sigma_n}\big]^{\frac{1}{2}}\one_F\one_{\{\sigma_n <\sigma\}}P[F^c\,|\,\F_{\sigma_n}]^{\frac{1}{2}}\right)\\
&{}=
\limsup_{n\to\infty}\left(\E\big[{}^{\sigma_n}\!Z_T^2\,\big|\,\F_{\sigma_n}\big]^{\frac{1}{2}}\right)\one_F\one_{F^c}\leq C\one_F\one_{F^c}=0,
\end{align*}
where $C=\sup_{n\in\N}\E\big[{}^{\sigma_n}\!Z_T^2\,\big|\,\F_{\sigma_n}\big]$. This yields $P(F)=0$ and hence $\tauZ_{\sigma-}\neq 0$, which means that each $\tauZ$ can only jump to zero. Therefore, $\tauZ_-\ne0$ on $\llbracket0,\sigma\rrbracket$ and $\tauZ=0$ on $\llbracket\sigma,T\rrbracket$ so that its stochastic logarithm $\Log(\tauZ)=\frac{\mathbbm{1}_{\llbracket0,\sigma\rrbracket}}{\tauZ_-}\sint \tauZ$ is well defined and gives $\tauZ=\Exp(\Log(\tauZ))$; see \cite[Proposition 2.2]{choulli.al.98}. Note that, since $\tauZ$ is a locally square-integrable martingale, $\Log(\tauZ)$ is a locally square-integrable martingale, too. 

Let $(T_n)_{n=0}^\infty$ be the sequence of stopping times given by $T_0=0$ and 
$$T_{n+1}=\inf\{t>T_{n}\,|\,{}^{T_n}\!Z_t=0\}\wedge T,\qquad n\in\N.$$ 
Then, because each ${}^{T_n}\! Z$ can only jump to zero, the sequence $(T_n)_{n=0}^\infty$ of stopping times converges stationarily to $T$. Since $\Log({}^{T_n}\!Z)_t(\omega)=0$ for $t\notin(T_n(\omega), T_{n+1}(\omega)]$ and the sequence $(T_n)_{n=0}^\infty$ of stopping times converges stationarily to $T$, defining $N:=\sum_{n=0}^\infty\Log({}^{T_n}\!Z)$ yields a locally square-integrable martingale. Fix $\tau\in\T$. Then, for each $\omega\in\Omega$, there is only one $n(\omega)\in\N\cup\{0\}$ such that $\tau(\omega)\in[ T_{n(\omega)}(\omega),T_{n(\omega)+1}(\omega))$. Observe that the mapping $\omega\mapsto n(\omega)$ is $\F_\tau$-measurable, since $\{n(\omega)=k\}=\{T_k\leq \tau< T_{k+1}\}\in\F_\tau$ for all $k\in\N\cup\{0\}$. Combining the time consistency of $\tauZ$ (property \ref{spd.2} of Definition~\ref{D:SPDlop}) with Yor's formula for stochastic exponentials and the definition of $N$ yields 
$$\tauZ_T(\omega)=\frac{{}^{T_{n(\omega)}}\!Z_T(\omega)}{{}^{T_{n(\omega)}}\!Z_{\tau}(\omega)}=\frac{\Exp({}^{T_{n(\omega)}}\!N)_T(\omega)}{\Exp({}^{T_{n(\omega)}}N)_{\tau})(\omega)}=\Exp(\tauN)_T(\omega)=\tauE(N)_T(\omega)$$
and hence $\tauZ_T=\tauE(N)_T$. By the square integrability of $\{\tauZ_T\}_{\tau\in\T}$ (property \ref{spd.2} of Definition~\ref{D:SPDlop}), the latter also implies that $\E[\tauE(N)_T^2\,|\,\F_\tau]=\E[(\tauZ_T)^2\,|\,\F_\tau]<\infty$ for all $\tau\in\T$ so that $\Exp(N)$ is a square-integrable $\Exp$-density (Definitions~\ref{D:Emart} and \ref{D:Edens}).\medskip

\noindent ``\ref{sopZ.3} $\Rightarrow$ \ref{sopZ.1}'': We only need to verify that setting $\tauZ_T:=\tauE(N)_T$ for $\tau\in\T$ defines a state price density $\{\tauZ_T\}_{\tau\in\T}$ satisfying the law of one price. Indeed, the family $\{\tauE(N)\}_{\tau\in\T}$ being an $\Exp$-density (as in Definition \ref{D:Edens}) implies the correct pricing of the risk-free asset (property \ref{spd.1} of Definition~\ref{D:SPDlop}) as well as the time consistency (property \ref{spd.2} of Definition~\ref{D:SPDlop}) by applying Yor's formula for stochastic exponentials. The conditional square integrability (property \ref{spd.3} of Definition~\ref{D:SPDlop}) follows directly from the fact that the $\Exp$-density $\{\tauE(N)\}_{\tau\in\T}$ is square integrable (as in Definition~\ref{D:Edens}). We prove the bounded conditional second moments at predictable stopping times (property \ref{spd.4} of Definition~\ref{D:SPDlop}) by contradiction. For this, let $\tau\in\T$ be a predictable stopping time and $(\tau_n)_{n=1}^\infty$ be an announcing sequence of $\tau$ such that $P(F_1)>0$ for $F_1:=\{\sup_{n\in\N}\E[{}^{\tau_n}\!\Exp(N)_T^2\,|\,\F_{\tau_n}]=\infty\}$. Note that $F_1\subseteq\{\tau>0\}$ by the conditional square integrability (property \ref{spd.3} of Definition~\ref{D:SPDlop}). Let $(T_m)_{m=0}^\infty$ be the sequence of stopping times given by $T_0=0$ and 
$$T_{m+1}=\inf\big\{t>T_{m}\,|\,\TmE(N)_t=0\big\},\qquad m\in\N.$$ 
Because each $\TmE(N)$ can only jump to zero, the sequence $(T_m)_{m=0}^\infty$ converges stationarily to $\infty$. Therefore, there is $k\in\N$ such that $P(F_2)>0$ with $F_2:=\{T_{k}<\tau\leq T_{k+1}\}\cap F_1$. On $F_2$, we have for sufficiently large $n$ that $\TkE(N)_T=\TkE(N)_{\tau_n}{}^{\tau_n}\!\Exp(N)_T$. Since $\E[\TkE(N)_T^2\,|\,\F_{T_k}]<\infty$, this yields that $\lim_{n\to\infty}\TkE(N)_{\tau_n}=\TkE(N)_{\tau-}=0$ on $F_2$. The latter contradicts the fact that $\TkE(N)_{\tau-}\ne0$ on $\{T_{k}<\tau\leq T_{k+1}\}$ by the definition of $T_{k+1}$, which completes the proof. 
\end{proof}
\begin{remark}
The proof of Proposition \ref{P:sopZ} shows that the condition \ref{sopZ.1} can be written in the following equivalent form.
\begin{enumerate}
\item[(i')]\label{sopZ.2}  $\{\tauZ_T\}_{\tau \in \T}$ satisfies properties \ref{spd.1}--\ref{spd.3} of Definition~\ref{D:SPDlop} and $\tauZ=(\tauZ_t)_{0\leq t\leq T}$ does not reach zero continuously and is absorbed in zero in the sense of \eqref{eq:dnrzcaz} for each $\tau \in \T$. 
\end{enumerate}
\end{remark}


\subsection{Proof of Theorem \ref{T:M1}}
We shall now construct a specific family $\{\tauhZ_T\}_{\tau\in\T}$ of variance-optimal state price densities as discussed in Remark \ref{rem:vo}.
\begin{lemma}\label{lem:mart}
For any $[0,T]$-valued stopping time $\tau$, define the square-integrable martingales $\hat{M}^{(\tau)}=(\hat{M}^{(\tau)}_t)_{0\leq t\leq T}$ by
\begin{equation}
\hat{M}^{(\tau)}_t:=\E\big[\big(1-\hat{G}^{(\tau)}_T\big)\,|\,\mathcal{F}_t\big],\quad 0\leq t\leq T.\label{lem:def:M}
\end{equation}
Then, the process
\begin{equation}
\hat{M}^{(\tau)}(x+\vt\sint S)=\left(\hat{M}^{(\tau)}_t(x+\vt\sint S_t)\right)_{0\leq t\leq T}\label{lem:mart:eq1}
\end{equation}
is a martingale for any $x\in L^2(\F_0,P)$ and $\vt\in \TB_\tau(0)$ and
\begin{equation}
L_{\tau}=\E\Big[\big(1-\hat{G}^{(\tau)}_T\big)^2\,\big|\,\mathcal{F}_{\tau}\Big]=\hat{M}^{(\tau)}_{\tau}.
\label{eq:rep:op:2}
\end{equation}
\end{lemma}
\begin{proof}
The martingale property of \eqref{lem:mart:eq1} follows from the first order condition of optimality. Indeed, $\vt\in\Theta_\tau$ implies that $\one_{F\times(s,t]}\vt$ is in $\Theta_\tau$ for all $s\leq t$ and arbitrary $F\in\F_s$. Therefore, since $\hat{G}^{(\tau)}_T$ is the orthogonal projection of $1$ onto $\cl\,\{\vt\sint S_T\,:\,\vt\in\Theta_\tau\}$, we have that
$$
\E\big[\one_F\big(\vt\sint S_t-\vt\sint S_s\big)\big(1-\hat{G}^{(\tau)}_T\big)\big]
=\E\big[ \big( (\one_{F\times(s,t]}\vt)\sint S_T\big) \big(1-\hat{G}^{(\tau)}_T\big) \big]=0.
$$
As $F\in\F_s$ was arbitrary, the latter yields, by the tower property of conditional expectations, that
\begin{align*}
\E\big[(\vt\sint S_t)\hat{M}^{(\tau)}_t\,\big|\,\F_s\big]&{}=\E\big[(\vt\sint S_t)(1-\hat{G}^{(\tau)}_T)\,\big|\,\F_s\big]\\
&{}=\E\big[(\vt\sint S_s)(1-\hat{G}^{(\tau)}_T)\,\big|\,\F_s\big]=(\vt\sint S_s)\hat{M}^{(\tau)}_s.
\end{align*}
The case $\vt\in\TB_\tau$ then follows by approximating $\vt\in\TB_\tau$ by a sequence $(\vt^n)_{n=1}^\infty$ of strategies $\vt^n=(\vt^n_t)_{0\leq t\leq T}\in\Theta_\tau$. This yields
\begin{align*} 
\E\big[(\vt\sint S_T)\hat{M}^{(\tau)}_T\,\big|\,\F_t\big]&=\E\left[\left(\lim_{n\to\infty}\vt^n\sint S_T\right)\hat{M}^{(\tau)}_T\,\big|\,\F_t\right]\\
&=\lim_{n\to\infty}\E\big[(\vt^n\sint S_T)\hat{M}^{(\tau)}_T\,\big|\,\F_t\big]\\
&=\lim_{n\to\infty}(\vt^n\sint S_t)\hat{M}^{(\tau)}_t=(\vt\sint S_t)\hat{M}^{(\tau)}_t,
\end{align*}
where we have made use of $\lim_{n\to\infty}\vt^n\sint S_T\xrightarrow{L^2}\vt\sint S_T$ in the first and second equality and $\lim_{n\to\infty}\vt^n\sint S_t\xrightarrow{P}\vt\sint S_t$ in the last equality.

Finally, let $(\vp^n)_{n=1}^\infty$ be a sequence of trading strategies $\vp^n=(\vp^n_t)_{0\leq t\leq T}\in\TB_\tau$ such that $\vp^n\sint S_T\xrightarrow{L^2}\hat{G}^{(\tau)}_T$. Then,
\begin{align*}
L_{\tau}&=\E\left[\left(1-\hat{G}^{(\tau)}_T\right)^2\,\Big|\,\mathcal{F}_{\tau}\right]\\
&{}=\E\left[1-\hat{G}^{(\tau)}_T\,\big|\,\mathcal{F}_{\tau}\right]+\E\left[\left(1-\hat{G}^{(\tau)}_T\right)\left(\lim_{n\to\infty}\vp^n\sint S_T\right)\,\Big|\,\mathcal{F}_{\tau}\right]\\
&{}=\E\left[1-\hat{G}^{(\tau)}_T\,\Big|\,\mathcal{F}_{\tau}\right]+\lim_{n\to\infty}\E\left[\left(1-\hat{G}^{(\tau)}_T\right)(\vp^n\sint S_T)\,\Big|\,\mathcal{F}_{\tau}\right]=\E\left[1-\hat{G}^{(\tau)}_T\,|\,\mathcal{F}_{\tau}\right].\tag*{\qedhere}
\end{align*}
\end{proof}
Using \eqref{eq:rep:op:2}, we can define processes $\hat{G}^{(\tau)}=(\hat{G}^{(\tau)}_t)_{0\leq t\leq T}$ as a substitute for the stochastic integrals $(\vp^{(\tau)}(1,0)\sint S_t)_{0\leq t\leq T}$.  Note, however, that our definition \eqref{def:G} below needs the weak  LOP condition that $L>0$ in order to recover the substitute wealth process $\hat{G}^{(\tau)}=(\hat{G}^{(\tau)}_t)_{0\leq t\leq T}$ from its terminal value $\hat{G}^{(\tau)}_T$.
\begin{lemma}\label{lem:mult}
Suppose that $L>0$. For each $\tau\in\T$, define the processes $\hat{G}^{(\tau)}=(\hat{G}^{(\tau)}_t)_{0\leq t\leq T}$ by 
\begin{equation}
\hat{G}^{(\tau)}_t=\begin{cases}0,&0\leq t<\tau,\\
1-\frac{\hat{M}^{(\tau)}_t}{L_t},&\tau\leq t\leq T.\end{cases}\label{def:G}
\end{equation}
Then, for all stopping times $\T\ni\sigma\geq\tau $, we have that
\begin{equation}
(1-\hat{G}^{(\tau)}_T)=(1-\hat{G}^{(\tau)}_\sigma)(1-\hat{G}^{(\sigma)}_T).\label{lem:mult:eq}
\end{equation}
\end{lemma}
\begin{proof}
Let $(\vt^n)_{n=1}^\infty$ be a sequence of trading strategies $\vt^n=(\vt^n_t)_{0\leq t\leq T}\in\Theta_\tau$ such that $\vt^n\sint S_T\to \hat{G}^{(\tau)}_T$ in $L^2$. Then,
\begin{equation}
(1-\vt^n\sint S_\sigma)^2L_\sigma\leq \E[(1-\vt^n\sint S_T)^2\,|\,\F_{\sigma}]\label{pr:eq1}
\end{equation}
by \cite[Proposition~3.1]{czichowsky.schweizer.13}. Since $(\vt^n\sint S_T)_{n=1}^\infty$ is a convergent sequence in $L^2(P)$, it is bounded in $L^2(P)$ and hence $(\vt^n\sint S_\sigma)_{n=1}^\infty$ is bounded in $L^2(P^\sigma)$ by \eqref{pr:eq1}, where $P^\sigma\sim P$ is defined by $\frac{dP^\sigma}{dP}=\frac{L_{\sigma}}{\E[L_{\sigma}]}>0$. By Mazur's Lemma \cite[Corollary~3.8]{brezis.11}, there exists a sequence $(\vp^n)_{n=1}^\infty$ of trading strategies $\vp^n\in\conv(\vt^n,\vt^{n+1},\ldots)\subseteq\Theta_\tau$ and a random variable $X_\sigma\in L^2(\F_\sigma,P^\sigma)$ such that $\vp^n\sint S_\sigma\to X_\sigma$ in $L^2(P^\sigma)$ and hence
\begin{equation}\label{pr:eq2}
(1-X_\sigma)^2L_\sigma\leq \E[(1-\hat{G}^{(\tau)}_T)^2\,|\,\F_{\sigma}].
\end{equation}

Note that \eqref{pr:eq2} implies that $\left(1-\left(X_\sigma+(1-X_\sigma)\hat{G}^{(\sigma)}_T\right)\right)=(1-X_\sigma)(1-\hat{G}^{(\sigma)}_T)\in L^2(P)$ and hence $Y_\sigma:=X_\sigma+(1-X_\sigma)\hat{G}^{(\sigma)}_T\in L^2(P)$. If we can show that
\begin{equation}\label{pr:eq3}
Y_\sigma=X_\sigma+(1-X_\sigma)\hat{G}^{(\sigma)}_T\in \mathcal{G}_\tau,
\end{equation}
then it follows from \eqref{pr:eq2} that $Y_\sigma$ is optimal for \eqref{A} and hence
\begin{equation}
Y_\sigma=\hat{G}^{(\tau)}_T.\label{pr:eq:2a}
\end{equation}

For the proof of \eqref{pr:eq3}, let $(\psi^n)_{n=1}^\infty$ be a sequence of trading strategies $\psi^n=(\psi^n_t)_{0\leq t\leq T}\in\Theta_\sigma$ such that $\psi^n\sint S_T\to \hat{G}^{(\sigma)}_T$ in $L^2(P)$. By Egorov's Theorem \cite[10.38]{aliprantis.border.06}, there exists a sequence $(F_m)_{m=1}^\infty$ of sets $F_m\in\F_\sigma$ with $P(F_m)\geq 1-1/m$ such that, for each $m\in\N$, we have that $(\vp^n\sint S_\sigma)_{n=1}^\infty$ and $X_\sigma$ are uniformly bounded on $F_m$ and $\vp^n\sint S_\sigma\to X_\sigma$ uniformly on $F_m$. Then, $(\xi^{m,n})_{n=1}^\infty$ given by 
$$\xi^{m,n}=\vp^n\mathbbm{1}_{F_m^c}+\left(\vp^n\mathbbm{1}_{\llbracket0,\sigma\rrbracket}+(1-\vp^n\sint S_\sigma)\psi^n\mathbbm{1}_{\rrbracket\sigma,T\rrbracket}\right)\mathbbm{1}_{F_m}$$
is a sequence of trading strategies $\xi^{m,n}=(\xi^{m,n}_t)_{0\leq t\leq T}\in\Theta(0,\tau)$ such that
$$\xi^{n,m}\sint S_T=\vp^n\sint S_T\mathbbm{1}_{F_m^c}+(1-\vp^n\sint S_{\sigma})(\psi^n\sint S_T)\mathbbm{1}_{F_m}$$
by the local character of the stochastic integrals. Hence, for each $m\in\N$,
$$\xi^{m,n}\sint S_T\xrightarrow{L^2(P)} \hat{G}^{(\tau)}_T\mathbbm{1}_{F_m^c}+(1-X_\sigma)\hat{G}^{(\sigma)}_T\mathbbm{1}_{F_m},\text{ as $n\to\infty$.}$$
Because $\hat{G}^{(\tau)}_T\mathbbm{1}_{F_m^c}+(1-X_\sigma)\hat{G}^{(\sigma)}_T\mathbbm{1}_{F_m}\xrightarrow{L^2(P)}(1-X_\sigma)\hat{G}^{(\sigma)}_T$, as $m\to\infty$, we can select a diagonal sequence $(\xi^{m,n_m})_{m=1}^\infty$ such that $\xi^{m,n_m}=(\xi^{m,n_m}_t)_{0\leq t\leq T}\in\Theta(0,\tau)$ and
$$\xi^{m,n_m}\sint S_T\xrightarrow{L^2(P)}Y_\sigma=X_\sigma+(1-X_\sigma)\hat{G}^{(\sigma)}_T,\text{ as $m\to\infty$,}$$ and hence \eqref{pr:eq3} holds.

From \eqref{pr:eq3}, we obtain
\begin{align*}
\hat{G}^{(\tau)}_\sigma&{}=1-\frac{\hat{M}^{(\tau)}_\sigma}{L_\sigma}\\
&{}=1-\E\left[(1-X_\sigma)(1-\hat{G}^{(\sigma)}_T)\,|\,\F_\sigma\right]\\
&{}=1-(1-X_\sigma)\E\left[(1-\hat{G}^{(\sigma)}_T)\,|\,\F_\sigma\right]=X_\sigma
\end{align*}
so that $X_\sigma$ is uniquely determined as $X_\sigma=\hat{G}^{(\tau)}_\sigma$ and we have \eqref{lem:mult:eq} by \eqref{pr:eq:2a}.
\end{proof}
\begin{lemma}\label{lem:hit0}
Suppose that $L>0$. For each $\tau\in\T$, let
\begin{equation}\label{eq:tauhZ}
\tauhZ_t=\begin{cases}1,&0\leq t<\tau,\\
\frac{\hat{M}^{(\tau)}_t}{L_\tau}=\frac{L_t(1-\hat{G}^{(\tau)}_t)}{L_\tau},&\tau\leq t\leq T.
\end{cases}
\end{equation}
Then, the family $\{\tauhZ_T\}_{\tau\in\T}$ of random variables is a state price density satisfying properties \ref{spd.1}--\ref{spd.3} of Definition~\ref{D:SPDlop} and compatible with $S$ such that $\E[(\tauhZ_T)^2\,|\,\F_\tau]=\frac{1}{L_\tau}$ for all $\tau\in\T$. If in addition $L_->0$, then $\{\tauhZ_T\}_{\tau\in\T}$ satisfies the law of one price, that is, $\{\tauhZ_T\}_{\tau\in\T}$ also fulfills property \ref{spd.4} of Definition~\ref{D:SPDlop}.
\end{lemma}
\begin{proof} We begin by verifing that the family $\{\tauhZ_T\}_{\tau\in\T}$ of random variables is a state price density satisfying properties \ref{spd.1}--\ref{spd.3} of Definition~\ref{D:SPDlop}. The definition \eqref{lem:def:M} of $\hat{M}^{(\tau)}$ together with \eqref{eq:rep:op:2} implies that 
$$\E[\tauhZ_T\,|\,\F_\tau]=\E\left[\frac{\hat{M}^{(\tau)}_T}{L_\tau}\,\Big|\,\F_\tau\right]=\frac{\hat{M}^{(\tau)}_\tau}{\hat{M}^{(\tau)}_\tau}=1\quad\text{for all $\tau\in\T$}$$
and hence the correct pricing of the risk-free asset (property \ref{spd.1} of Definition~\ref{D:SPDlop}). The time consistency (property \ref{spd.2} of Definition~\ref{D:SPDlop}) follows from the fact that $\tauhZ_t=(1-\hat{G}^{(\tau)}_t)/L_\tau$ by definition together with \eqref{lem:mult:eq}. Combining that $\tauhZ_T=(1-\hat{G}^{(\tau)}_T)/L_\tau$ with $\E[(1-\hat{G}^{(\tau)}_T)^2|\mathcal{F}_{\tau}]=L_{\tau}<\infty$ for all $\tau\in\T$ by \eqref{eq:rep:op:2}, we obtain that $\tauhZ_T$ is conditionally square integrable (as in property \ref{spd.3} of Definition~\ref{D:SPDlop}) and $\E[(\tauhZ_T)^2|\F_\tau]=\frac{1}{L_\tau}$ for all $\tau\in\T$. Because $\mathbbm{1}_{\rrbracket 0,\sigma_n\wedge\sigma\rrbracket}\in\Theta\subseteq\TB$ is a simple trading strategy for any localizing sequence of stopping times $(\sigma_n)_{n=1}^\infty $ such that $S_{\sigma_n}^*\in L^2$, we have that $S$ is compatible with $\{\tauZ_T\}_{\tau\in\T}$ (in the sense of Definition \ref{D:SPDlop}) by choosing $\tau=0$, $x=S_0$ and $\vt=\mathbbm{1}_{\rrbracket 0,\sigma_n\wedge\sigma\rrbracket}$ in \eqref{lem:mart:eq1}.

For the proof of the bounded conditional expectations of squares at predictable stopping times (property \ref{spd.4} of Definition~\ref{D:SPDlop}), suppose that $L_->0$. Let $\sigma\in\T$ be a predictable stopping time and $(\sigma_n)_{n=1}^\infty$ be an announcing sequence of stopping times of $\sigma$. Because
$$\E[({}^{\sigma_n}\hat{Z}_T)^2\,|\,\F_{\sigma_n}]=\E\left[\left(\frac{1-\hat{G}^{(\sigma_n)}_T}{L_{\sigma_n}}\right)^2\,\bigg|\,\F_{\sigma_n}\right]=\frac{1}{L_{\sigma_n}},$$we obtain that $\sup_{n\in\N}\E[({}^{\sigma_n}\hat{Z}_T)^2\,|\,\F_{\sigma_n}]=\sup_{n\in\N}\frac{1}{L_{\sigma_n}}<\infty$, since $L_{\sigma_n}$ converges to $L_{\sigma-}$ $P$-a.s.~ and we have both $L>0$ and $L_->0$.
\end{proof}

\begin{proof}[Proof of Theorem \ref{T:M1}] The most efficient way to prove the theorem is to show ``\ref{MT1.1} $\Rightarrow$ \ref{MT1.5}'', ``\ref{MT1.5} $\Rightarrow$ \ref{MT1.4}'', ``\ref{MT1.4} $\Rightarrow$ \ref{MT1.3}'', ``\ref{MT1.3} $\Rightarrow$ \ref{MT1.2}'', and ``\ref{MT1.2} $\Rightarrow$ \ref{MT1.1}''.

``\ref{MT1.1} $\Rightarrow$ \ref{MT1.5}'': For a proof by contradiction, we first suppose that $L>0$ fails. Then, we have $P(F)>0$ for $F:=\{ \tau<\infty\}$ and the stopping time
$$\tau:=\inf\{ t>0\,:\,L_t=0\}$$
and therefore $0=\mathbbm{1}_FL_\tau=\mathbbm{1}_F\left(\essinf_{\vt\in\TB_\tau}\E\left[(1-\vt\sint S_T)^2\,|\,\F_\tau\right]\right)$. Because the family
$$\Gamma=\left \{\E\left[(1-\vt\sint S_T)^2\,|\,\F_\tau\right]\,:\,\vt\in\TB_\tau\right\}$$
of random variables is stable under taking minima by \cite[Lemma 2.18(1)]{czichowsky.schweizer.13}, there exists a sequence $(\vt^n)_{n=1}^\infty$ of trading strategies $\vt^n=(\vt^n_t)_{0\leq t\leq T}\in\TB_\tau$ such that 
$$\E\left[\left(\mathbbm{1}_F-\mathbbm{1}_F(\vt^n\sint S_T)\right)^2\right]=\E\left[\mathbbm{1}_F\E\left[(1-\vt^n\sint S_T)^2\,|\,\F_\tau\right]\right]\to 0.$$
As $0$ and $\vt^n$ are in $\TB_\tau$ for all $n\geq 1$, we have by \cite[Lemma 2.18(1)]{czichowsky.schweizer.13} that $\psi^n:=(\mathbbm{1}_{F}\vt^n)\in\TB_\tau$ for all $n\geq 1$. Therefore, we have a sequence $(\psi^n)_{n=1}^\infty$ of trading strategies $\psi^n\in\TB_\tau$ and $F\in\F_\tau$ such that $\mathbbm{1}_{F}(1-\psi^n\sint S_T)\xrightarrow{L^2}0$, but $\lim_{n\to\infty}\mathbbm{1}_{F}(1-\psi^n\sint S_\tau)=\mathbbm{1}_F\ne0$. Approximating the strategies $\psi^n\in\TB_\tau$ by simple trading strategies $\vp^n\in \TB_{\tau}$, this contradicts \ref{def:lop:1} of Definition \ref{D:LOP1S} of the law of one price. 

If $L_->0$ fails, then $P(F)>0$ for $F:=\{\sigma<\infty\}$ and the predictable stopping time $\sigma=\inf\{ t>0\,:\,L_{t-}=0\}$. Let $(\sigma_n)_{n=1}^\infty$ be an announcing sequence of stopping times of $\sigma$ and set $F_n=\{ \sigma_n<\infty\}$. Then, $\mathbbm{1}_{F_n}\xrightarrow{L^0}\mathbbm{1}_{F}$ and $\mathbbm{1}_{F_n}L_{\sigma_n}\xrightarrow{L^2}\mathbbm{1}_FL_{\sigma-}$. Since each of the families
$$\Gamma_n=\left \{\E\left[(1-\vt\sint S_T)^2\,|\,\F_{\sigma_n}\right]\,:\,\vt\in\TB_{\sigma_n}\right\}$$
of random variables is stable under taking minima by \cite[Lemma 2.18(1)]{czichowsky.schweizer.13} for each $n\geq 1$, there exists a diagonal sequence $(\vt^n)_{n=1}^\infty$ of trading strategies $\vt^n=(\vt^n_t)_{0\leq t\leq T}\in\TB_{\sigma_n}$ such that 
$$\E\left[\left(\mathbbm{1}_{F_n}-\mathbbm{1}_{F_n}(\vt^n\sint S_T)\right)^2\right]=\E\left[\mathbbm{1}_{F_n}\E\left[(1-\vt^n\sint S_T)^2\,|\,\F_{\sigma_n}\right]\right]\to 0.$$ 
Because $0$ and $\vt^n$ are in $\TB_{\sigma_n}$ for all $n\geq 1$, we have again by \cite[Lemma~2.18(1)]{czichowsky.schweizer.13} that $\psi^n:=(\mathbbm{1}_{F_n}\vt^n)\in\TB_{\sigma_n}$ for all $n\geq 1$. This gives a sequence $(\psi^n)_{n=1}^\infty$ of trading strategies $\psi^n\in\TB_{\sigma_n}$ and $F_n\in\F_{\sigma_n}$ such that $\mathbbm{1}_{F_n}(1-\psi^n\sint S_T)\xrightarrow{L^2}0$, but $\lim_{n\to\infty}\mathbbm{1}_{F_n}(1-\psi^n\sint S_\tau)=\mathbbm{1}_F\ne0$ and therefore yields a contradict to \ref{def:lop:2} of Definition \ref{D:LOP1S} of the law of one price by approximating the strategies $\psi^n\in\TB_{\sigma_n}$ by simple trading strategies $\vp^n\in \TB_{\sigma_n}$.\smallskip\\
\smallskip ``\ref{MT1.5} $\Rightarrow$ \ref{MT1.4}'': This follows directly from Proposition \ref{P:sopZ} and Lemma \ref{lem:hit0}.\smallskip\\
``\ref{MT1.4} $\Rightarrow$ \ref{MT1.3}'': Since $\Exp(N)$ is a square-integrable $\Exp$-density (Definition~\ref{D:Edens}), Proposition \ref{P:sopZ} yields that that setting $\tauZ_T:=\tauE(N)_T$ for $\tau\in\T$ gives a state price density $\{\tauZ_T\}_{\tau\in\T}$ satisfying LOP (Definitions~\ref{D:SPDlop}). Moreover, the property that $S$ is an $\Exp(N)$-local martingale (Definitions~\ref{D:Emart}) directly implies that $\{\tauZ_T\}_{\tau\in\T}$ is compatible with $S$ (Definition~\ref{D:SPDlop}). \smallskip\\
``\ref{MT1.3} $\Rightarrow$ \ref{MT1.2}'': Given a compatible state price density $\{\tauZ_T\}_{\tau\in\T}$ satisfying LOP, we define a family $\{p_\tau\}_{\tau\in\T}$ of operators $p_\tau: L^2(\F_T,P)\to L^0(\F_\tau,P)$ by
\begin{equation*}
\text{$p_{\tau}(H)=\E[\tauZ_T\mkern0.1mu H\,|\,\F_\tau]$ for all $H\in L^2(\F_T,P)$.}\label{pr:main:eq:p}
\end{equation*}
Since $\{\tauZ_T\}_{\tau\in\T}$ satisfies LOP and is compatible with $S$ (Definition~\ref{D:SPDlop}), it is straightforward to check that $\{p_\tau\}_{\tau\in\T}$ is a compatible price system satisfying LOP (Definition~\ref{D:PSlop}) by comparing the definitions using the implication ``\ref{ubp3}$\ \Rightarrow$\ \ref{ubp1}'' of Proposition \ref{P:UBP}.\smallskip\\
``\ref{MT1.2} $\Rightarrow$ \ref{MT1.1}'':
For a proof by way of contradiction, suppose that \ref{MT1.1} and therefore either condition \ref{def:lop:1} or condition \ref{def:lop:2} of Definition \ref{D:LOP1S} fail.

We begin with condition \ref{def:lop:1}. If condition \ref{def:lop:1} fails, there exist a stopping time $\tau\in\T$, an $\F_\tau$-measurable endowment $x_\tau$ and a sequence $(\vt^n)_{n=1}^\infty$ of simple trading strategies such that $x_\tau+\vt^n\one_{\rrbracket\tau,T\rrbracket}\sint S_T\xrightarrow{L^2}0$ and  $x_\tau\ne0$. By the conditional linearity (Property \ref{D:PSlop:3a} of Definition \ref{D:PSlop}) and the time consistency for simple strategies, we have that $p_\tau(x_\tau+\vt^n\one_{\rrbracket\tau,T\rrbracket}\sint S_T)=x_\tau$ and hence $x_\tau=p_\tau(x_\tau+\vt^n\one_{\rrbracket\tau,T\rrbracket}\sint S_T)\xrightarrow{L^0}0$ by the continuity and conditional linearity of $p_\tau$ (Properties \ref{D:PSlop:3b} and \ref{D:PSlop:3a} of Definition \ref{D:PSlop}). This, however, is a contradiction to $x_\tau\ne0$.
Similarly, if condition \ref{def:lop:2} fails, there exist a predictable stopping $\sigma\in\T$, an announcing sequence $(\sigma_n)_{n=1}^\infty$ of stopping times for $\sigma$, sequences $(x^n_{\sigma_n})_{n=1}^\infty$ of $\F_{\sigma_n}$-measurable random variables and $(\vt^n)_{n=1}^\infty$ of simple trading strategies such that 
$x^n_{\sigma_n}+\vt^n\one_{\rrbracket\sigma_n,T\rrbracket}\sint S_T\xrightarrow{L^2}0$ and 
$x^n_{\sigma_n}\xrightarrow{L^0} x_{\sigma-}$ for some random variable $x_{\sigma-}$ with $x_{\sigma-}\ne0$. Then, we have, again by the conditional linearity (Property \ref{D:PSlop:3a} of Definition \ref{D:PSlop}) and the time consistency for simple strategies, that 
$p_{\sigma_n}(x^n_{\sigma_n}+\vt^n\one_{\rrbracket\sigma_n,T\rrbracket}\sint S_T)=x^n_{\sigma_n}$. By Definition \ref{D:PSlop}, $\{p_{\sigma_n}\}_{n=1}^\infty$ is a sequence of $\F_{\sigma_n}$-linear continuous  mappings $p_{\sigma_n}:L^2(\F_T,P)\to L^0(\F_{\sigma_n},P)$ that are pointwise convergent, that is, $(p_{\sigma_n}(H))_{n=1}^\infty$ is a convergent sequence in $L^0$ for all $H\in L^2(\F_T,P)$. Therefore, it follows from the implication ``\ref{ubp1}$\Rightarrow$\ref{ubp3}'' of Proposition \ref{P:UBP} that the conditional operator norms $\|p_{\sigma_n}\|$ are uniformly bounded, that is,
$$C:=\sup_{n\in\N}\|p_{\sigma_n}\|<\infty.$$ Combining the latter with $x^n_{\sigma_n}+\vt^n\one_{\rrbracket\sigma_n,T\rrbracket}\sint S_T\xrightarrow{L^2}0$ gives that
\setlength{\abovedisplayskip}{0pt}
\setlength{\belowdisplayskip}{0pt}
\[ 
|x_{\sigma_n}|=|p_{\sigma_n}(x^n_\tau+\vt^n\one_{\rrbracket\sigma_n,T\rrbracket}\sint S_T)|\leq \sup_{n\in\N}\|p_{\sigma_n}\| \left(\E\left[\left(x^n_{\sigma_n}+\vt^n\one_{\rrbracket\sigma_n,T\rrbracket}\sint S_T\right)^2\,\big|\,\F_{\sigma_n}\right]\right)^{\frac{1}{2}}\xrightarrow{L^0}0,
\]
which contradicts $x_{\sigma-}\ne0$ with $x^n_{\sigma_n}\xrightarrow{L^0} x_{\sigma-}$.\smallskip\\
\noindent \ref{MT1.6}: By \ref{MT1.4}, there exists a semimartingale $N$ such that $\Exp(N)$ is a square-integrable $\Exp$-density and $S$ is an $\Exp(N)$-local martingale. Therefore, the closedness of 
$\{\vt\sint S_T:\vt\in\TB\}$
follows by applying \cite[Theorem 2.16]{czichowsky.schweizer.13}.\smallskip\\
\noindent \ref{MT1.7}: Fix $\tau\in\T$ and let $\tilde{S}=\mathbbm{1}_{\rrbracket \tau,T\rrbracket }\sint S$. Observe that $\TB(\tilde{S})=\TBtau(S)$ and that $\vt\sint \tilde{S}=\vt\sint S$ for all $\vt\in\TB(\tilde{S})=\TBtau(S)$ by the associativity of stochastic integrals. Moreover, the properties \ref{MT1.1}--\ref{MT1.5} imposed on $S$ in Theorem \ref{T:M1} are inherited by $\tilde{S}$. By \ref{MT1.4}, there exists a semimartingale $N$ such that $\Exp(N)$ is a square-integrable $\Exp$-density and $S$ hence also $\tilde{S}$ is an $\Exp(N)$-local martingale. Therefore, the closedness of 
$\{\vt\sint S_T:\vt\in\TBtau\}=\{\tilde{\vt}\sint \tilde{S}_T\,:\,\tilde{\vt}\in\TB(\tilde{S})\}$ 
follows again by applying \cite[Theorem 2.16]{czichowsky.schweizer.13} to $\tilde{S}$.\smallskip\\
\noindent \ref{MT1.8}: Fix $\vt\in \TB_{\sigma-}$. Because $\one_{\llbracket 0,\sigma_n\rrbracket}\vt=0$ for all $n\in\N$ by definition of $\TB_{\sigma-}$, we have that $\vt\in\TB_{\sigma_n}$ for all $n\in\N$ and hence $\{\vt\sint S_T\,:\,\vt\in\TB_{\sigma-}\} \subseteq \cap_{n=1}^\infty\{\vt\sint S_T\,:\,\vt\in\TB_{\sigma_n}\}.$

Next, suppose that $\vt\in\TB_{\sigma_n}$ for all $n\in\N$. Then, $\one_{\llbracket 0,\sigma_n\rrbracket}\vt=0$ for all $n\in\N$ by definition of $\TB_{\sigma_n}$. Therefore, $(\one_{\llbracket 0,\sigma\llbracket}\vt)\sint S = \lim_{n\to\infty}(\one_{\llbracket 0,\sigma_n\rrbracket}\vt)\sint S=0$ in the semimartingale topology and hence $\vt\in\TB_{\sigma-}$ so that $  \cap_{n=1}^\infty\{\vt\sint S_T\,:\,\vt\in\TB_{\sigma_n}\}\subseteq \{\vt\sint S_T\,:\,\vt\in\TB_{\sigma-}\}.$\bigskip
\end{proof}
\begin{acknowledgement}
We thank two anonymous referees, associate editor, and co-editor for careful reading of the paper and helpful comments. 
\end{acknowledgement}
\def\shorturl#1{\href{http://#1}{\nolinkurl{#1}}}
\def\MR#1{\href{http://www.ams.org/mathscinet-getitem?mr=#1}{MR#1}}
\def\ARXIV#1{\href{https://arxiv.org/abs/#1}{arXiv:#1}}
\def\DOI#1{\href{https://doi.org/#1}{doi:#1}}


\begin{thebibliography}{10}

\bibitem{aliprantis.border.06}
C.~D. Aliprantis and K.~C. Border, \emph{Infinite {D}imensional {A}nalysis: {A}
  {H}itchhiker's {G}uide}, 3rd ed., Springer, Berlin, 2006. \MR{2378491}

\bibitem{battig.jarrow.99}
R.~J. B\"{a}ttig and R.~A. Jarrow, \emph{The second fundamental theorem of
  asset pricing: {A} new approach}, Rev. Financ. Stud. \textbf{12} (1999),
  no.~5, 1219--1235.

\bibitem{brezis.11}
H.~Brezis, \emph{Functional {A}nalysis, {S}obolev {S}paces and {P}artial
  {D}ifferential {E}quations}, Universitext, Springer, New York, 2011.
  \MR{2759829}

\bibitem{cerny.kallsen.07}
A.~{\v{C}}ern\'{y} and J.~Kallsen, \emph{On the structure of general
  mean--variance hedging strategies}, Ann. Probab. \textbf{35} (2007), no.~4,
  1479--1531. \MR{2330978}

\bibitem{cerny.kallsen.08a}
A.~{\v{C}}ern\'{y} and J.~Kallsen, \emph{A counterexample concerning the
  variance-optimal martingale measure}, Math. Finance \textbf{18} (2008),
  no.~2, 305--316. \MR{2395578}

\bibitem{cerny.kallsen.08b}
A.~{\v{C}}ern\'{y} and J.~Kallsen, \emph{Mean--variance hedging and optimal
  investment in {H}eston's model with correlation}, Math. Finance \textbf{18}
  (2008), no.~3, 473--492. \MR{2427731}

\bibitem{cv.al.16.jmaa}
S.~Cerreia-Vioglio, M.~Kupper, F.~Maccheroni, M.~Marinacci, and N.~Vogelpoth,
  \emph{Conditional {$L_p$}-spaces and the duality of modules over
  {$f$}-algebras}, J. Math. Anal. Appl. \textbf{444} (2016), no.~2, 1045--1070.
  \MR{3535749}

\bibitem{choulli.al.98}
T.~Choulli, L.~Krawczyk, and C.~Stricker, \emph{{${\Exp}$}-martingales and
  their applications in mathematical finance}, Ann. Probab. \textbf{26} (1998),
  no.~2, 853--876. \MR{1626523}

\bibitem{cochrane.01}
J.~H. Cochrane, \emph{{A}sset {P}ricing}, Princeton University Press, 2001.

\bibitem{courtault.al.04}
J.-M. Courtault, F.~Delbaen, Y.~Kabanov, and C.~Stricker, \emph{On the law of
  one price}, Finance Stoch. \textbf{8} (2004), no.~4, 525--530. \MR{2212116}

\bibitem{czichowsky.schweizer.13}
C.~Czichowsky and M.~Schweizer, \emph{Cone-constrained continuous-time
  {M}arkowitz problems}, Ann. Appl. Probab. \textbf{23} (2013), no.~2,
  764--810. \MR{3059275}

\bibitem{delbaen.al.97}
F.~Delbaen, P.~Monat, W.~Schachermayer, M.~Schweizer, and C.~Stricker,
  \emph{Weighted norm inequalities and hedging in incomplete markets}, Finance
  Stoch. \textbf{1} (1997), 181--227.

\bibitem{delbaen.schachermayer.94}
F.~Delbaen and W.~Schachermayer, \emph{A general version of the fundamental
  theorem of asset pricing}, Math. Ann. \textbf{300} (1994), no.~3, 463--520.
  \MR{1304434}

\bibitem{delbaen.schachermayer.96.bej}
F.~Delbaen and W.~Schachermayer, \emph{The variance-optimal martingale measure
  for continuous processes}, Bernoulli \textbf{2} (1996), no.~1, 81--105.
  \MR{1394053}

\bibitem{dellacherie.meyer.82}
C.~Dellacherie and P.-A. Meyer, \emph{Probabilities and {P}otential. {B}},
  North-Holland Mathematics Studies, vol.~72, North-Holland, Amsterdam, 1982,
  Translation of \emph{{Probablit\'{e}s et Potentiel B}}, Herman, Paris, 1980.
  \MR{745449}

\bibitem{doleans-dade.70}
C.~Dol\'{e}ans-Dade, \emph{Quelques applications de la formule de changement de
  variables pour les semimartingales}, Z. Wahrscheinlichkeitstheorie verw.
  Gebiete \textbf{16} (1970), 181--194. \MR{283883}

\bibitem{duffie.richardson.91}
D.~Duffie and H.~R. Richardson, \emph{Mean--variance hedging in continuous
  time}, Ann. Appl. Probab. \textbf{1} (1991), no.~1, 1--15. \MR{1097461}

\bibitem{hansen.jagannathan.91}
L.~P. Hansen and R.~Jagannathan, \emph{Implications of security market data for
  models of dynamic economies}, J. Polit. Econ. \textbf{99} (1991), no.~2,
  225--262.

\bibitem{hansen.richard.87}
L.~P. Hansen and S.~F. Richard, \emph{The role of conditioning information in
  deducing testable restrictions implied by dynamic asset pricing models},
  Econometrica \textbf{55} (1987), no.~3, 587--613. \MR{890855}

\bibitem{he.wang.yan.92}
S.~W. He, J.~G. Wang, and J.~A. Yan, \emph{{S}emimartingale {T}heory and
  {S}tochastic {C}alculus}, Science Press, Beijing, 1992. \MR{1219534}

\bibitem{js.03}
J.~Jacod and A.~N. Shiryaev, \emph{Limit {T}heorems for {S}tochastic
  {P}rocesses}, 2nd ed., Comprehensive Studies in Mathematics, vol. 288,
  Springer, Berlin, 2003. \MR{1943877}

\bibitem{koch-medina.munari.20}
P.~Koch-Medina and C.~Munari, \emph{Market-{C}onsistent {P}rices---{A}n
  {I}ntroduction to {A}rbitrage {T}heory}, Birkh\"{a}user/Springer, Cham, 2020.
  \MR{4233192}

\bibitem{liptser.shiryaev.01.part1}
R.~S. Liptser and A.~N. Shiryaev, \emph{Statistics of {R}andom {P}rocesses.
  {I}}, Applications of Mathematics (New York), vol.~5, Springer-Verlag,
  Berlin, 2001. \MR{1800857}

\bibitem{markowitz.52}
H.~Markowitz, \emph{Portfolio selection}, J. Finance \textbf{7} (1952), no.~1,
  77--91.

\bibitem{melnikov.nechaev.99}
A.~V. Melnikov and M.~L. Nechaev, \emph{On the mean--variance hedging of
  contingent claims}, Theory Probab. Appl. \textbf{43} (1999), no.~4, 588--603.
  \MR{1692421}

\bibitem{monat.stricker.95}
P.~Monat and C.~Stricker, \emph{F\"{o}llmer--{S}chweizer decomposition and
  mean--variance hedging for general claims}, Ann. Probab. \textbf{23} (1995),
  no.~2, 605--628. \MR{1334163}

\bibitem{pham.00}
H.~Pham, \emph{On quadratic hedging in continuous time}, Math. Methods Oper.
  Res. \textbf{51} (2000), no.~2, 315--339. \MR{1761862}

\bibitem{protter.05}
P.~E. Protter, \emph{Stochastic {I}ntegration and {D}ifferential {E}quations},
  2nd ed., Stochastic Modelling and Applied Probability, vol.~21, Springer,
  Berlin, 2005. \MR{2273672}

\bibitem{schachermayer.02}
W.~Schachermayer, \emph{No arbitrage: On the work of {D}avid {K}reps},
  Positivity \textbf{6} (2002), no.~3, 359--368. \MR{1932656}

\bibitem{schweizer.92.aap}
M.~Schweizer, \emph{Mean--variance hedging for general claims}, Ann. Appl.
  Probab. \textbf{2} (1992), no.~1, 171--179. \MR{1143398}

\bibitem{schweizer.96}
M.~Schweizer, \emph{Approximation pricing and the variance-optimal martingale
  measure}, Ann. Probab. \textbf{24} (1996), no.~1, 206--236. \MR{1387633}

\bibitem{schweizer.01}
M.~Schweizer, \emph{A guided tour through quadratic hedging approaches}, Option
  {P}ricing, {I}nterest {R}ates and {R}isk {M}anagement, Handb. Math. Finance,
  Cambridge Univ. Press, Cambridge, 2001, pp.~538--574. \MR{1848562}

\bibitem{schweizer.10}
M.~Schweizer, \emph{Mean--variance hedging}, Encyclopedia of Quantitative
  Finance (R.~Cont, ed.), Wiley, Chichester, 2010, pp.~1177--1180.

\bibitem{stricker.90}
C.~Stricker, \emph{Arbitrage et lois de martingale}, Ann. Inst. H. Poincar\'{e}
  Probab. Statist. \textbf{26} (1990), no.~3, 451--460. \MR{1066088}

\bibitem{tao.10}
T.~Tao, \emph{An {E}psilon of {R}oom, {I}: {R}eal {A}nalysis}, Graduate Studies
  in Mathematics, vol. 117, American Mathematical Society, Providence, RI,
  2010. \MR{2760403}

\end{thebibliography}
\end{document}